\documentclass[reqno]{amsart}
\usepackage{mathmacros}
\usepackage{xr}
\usepackage{hyperref}
\usepackage[tmargin=1in,lmargin=1in,bmargin=1in,rmargin=1in]{geometry}
\usepackage{comment}

\allowdisplaybreaks

\makeatletter
\renewcommand{\p@enumii}{}
\makeatother
  
\numberwithin{equation}{section}

\title[Potentials of quadratic growth]{Energy-critical NLS with potentials of quadratic growth}
\author{Casey Jao}
\address{Department of Mathematics, UC Berkeley}
\email{cjao@math.berkeley.edu}

\begin{document}


\begin{abstract}
  Consider the global wellposedness problem for nonlinear
  Schr\"{o}dinger equation
\[
i\partial_t u = [-\tfr{1}{2} \Delta + V(x)] u \pm |u|^{4/(d-2)} u, \ u(0)
\in \Sigma(\mf{R}^d),
\]
where $\Sigma$ is the weighted Sobolev space $\dot{H}^1 \cap |x|^{-1} L^2$. The case $V(x) = \tfr{1}{2}|x|^2$ was recently treated by the
author. This note generalizes the results to a class of
``approximately quadratic'' potentials.

We closely follow the previous concentration compactness arguments for
the harmonic oscillator. A key technical difference is that in
the absence of a concrete formula for the linear propagator, we
apply more general tools from microlocal analysis, including a Fourier
integral parametrix of Fujiwara.
\end{abstract}
 \maketitle

\section{Introduction}
\label{section:intro}

We consider the nonlinear Schr\"{o}dinger equation
\begin{equation}
\label{eqn:nls_approx_quadratic}
\left\{ \begin{array}{c} i\partial_t u = (-\tfr{1}{2}\Delta + V) u +
    \mu |u|^{\fr{4}{d-2}} u, \quad \mu = \pm 1, \\[2mm]
u(0) = u_0 \in \Sigma(\mf{R}^d),\end{array}\right.
\end{equation}
where $V = V(x)$ is a real-valued potential The equation is
\emph{defocusing} or \emph{focusing} if $\mu = 1$ or $\mu=-1$,
respectively.  In a recent work~\cite{me_quadratic_potential}, we
studied large-data global wellposedness of the Cauchy problem with the
harmonic oscillator potential $V(x) = \tfr{1}{2}|x|^2$, for which
$\Sigma := \dot{H}^1 \cap |x|^{-1}L^2$, the weighted Sobolev space with norm
$ \|f\|_{\Sigma}^2 := \| \nabla f\|_{L^2}^2 + \|xf\|_{L^2}^2 <
\infty$, is precisely the function space associated with the conserved
energy
\[
E(u(t)) = \int_{\mf{R}^d} \tfr{1}{2} |\nabla u(t, x)|^2 + V(x) |u(t,x)|^2
+ \mu(1- \tfr{2}{d}) |u(t,x)|^{\fr{2d}{d-2}} \, dx = E(u(0)).
\]

This note extends the previous results to a wider class of
potentials that grow approximately quadratically.  More precisely, we
assume that $V$ is smooth and satisfies
\begin{gather}
  \label{eqn:potential_hyp_0}
  \partial^\alpha_x V \in L^\infty \quad \text{for all }
  |\alpha| \ge 2,\\
  \label{eqn:potential_hyp_1}
  V(x) \ge \delta |x|^2 \quad \text{for some } \delta > 0.
\end{gather}
These hypotheses ensure that 
$
\delta|x|^2 \le V(x) \le \delta^{-1} (1 + |x|^2)
$
for some constant $\delta > 0$. Therefore, by Sobolev embedding $\Sigma$ is still the energy
space and is also the form domain $Q(H) = D(H^{1/2})$ for the positive
operator $H = -\tfr{1}{2}\Delta + V$. It will be convenient at times to
use the equivalent norm
\[
\| f\|_{Q(H)}^2 := \| H^{1/2} f\|_{L^2}^2 = \|\nabla f\|_{L^2}^2 + \|
V^{1/2} f\|_{L^2}^2,
\]
which is exactly preserved by the propagator $e^{-itH}$.

This equation is closely linked to the energy-critical NLS
\begin{gather}
\label{eqn:nls_wo_potential}
\begin{split}
(i\partial_t + \tfr{1}{2}\Delta)u = \mu
    |u|^{\fr{4}{d-2}} u, \quad u(0) \in \dot{H}^1(\mf{R}^d)\\
E_\Delta(u) =\int_{\mf{R}^d} \tfr{1}{2} |\nabla u|^2 + \mu(1 -
\tfr{2}{d}) |u|^{\fr{2d}{d-2}} \, dx,
\end{split}
\end{gather}
which is invariant under the scaling
$u \mapsto u^{\lambda}(t, x) = \lambda^{-\fr{d-2}{2}} u( \lambda^{-2}
t, \lambda^{-1} x)$.  Roughly speaking, if a solution $u$ to
\eqref{eqn:nls_approx_quadratic} is initially highly concentrated at
some point $x_0$, it sees the potential $V$ as approximately a
constant $V(x_0)$, and for short times the behavior of $u$ will be
modelled, up to a temporal phase, by equation
\eqref{eqn:nls_wo_potential}. 

As with the harmonic oscillator~\cite{me_quadratic_potential}, it will
be essential to formulate this approximation precisely and understand the
behavior of solutions to the limiting scale-invariant
equation. Fortunately, the latter problem has received considerable
attention in the past twenty years. We summarize the state of the art
in the following conjecture and theorem, which we employ as a black
box in our analysis:



\begin{conjecture}
\label{conjecture:nls_wo_potential_gwp}
  When $\mu = 1$, solutions to \eqref{eqn:nls_wo_potential} exist
  globally and scatter. That is, for any $u_0 \in \dot{H}^1(\mf{R}^d)$,
  there exists a unique global solution $u : \mf{R} \times \mf{R}^d \to
  \mf{C}$ to \eqref{eqn:nls_wo_potential} with $u(0) = u_0$, and
  this solution satisfies a spacetime bound
\begin{equation}
\label{eqn:nls_wo_potential_spacetime_bound}
S_{\mf{R}}(u) := \int_{\mf{R}} \int_{\mf{R}^d}
|u(t,x)|^{\fr{2(d+2)}{d-2}} \, dx \, dt \le C(E_{\Delta}(u_0)) < \infty.
\end{equation}
Moreover, there exist functions $u_{\pm} \in \dot{H}^1(\mf{R}^d)$ such that
\[
\lim_{t \to \pm \infty} \|u(t) - e^{\pm \fr{it\Delta}{2}} u_{\pm}
\|_{\dot{H}^1} = 0,
\]
and the correspondences $u_0 \mapsto u_{\pm}(u_0)$ are homeomorphisms
of $\dot{H}^1$. 

  When $\mu = -1$, one also has global
  wellposedness and scattering provided that 
\[
E_{\Delta}(u_0) < E_{\Delta}(W), \quad \| \nabla u_0\|_{L^2} < \|
\nabla W\|_{L^2},
\]
where the \emph{ground state}
\[
W(x) =  \Bigl ( 1 + \fr{2|x|^2}{d(d-2)} \Bigr)^{-\fr{d-2}{2}} \in
\dot{H}^1 (\mf{R}^d)
\]
solves the elliptic equation $\tfr{1}{2}\Delta + |W|^{\fr{4}{d-2}} W =
0$. 
\end{conjecture}


\begin{thm}
  \label{thm:nls_wo_potential_gwp}
Conjecture \ref{conjecture:nls_wo_potential_gwp} holds for the defocusing
equation. For the focusing equation, the conjecture holds for radial
initial data when $d \ge 3$, and for all initial data when $d \ge 5$. 
\end{thm}
\begin{proof}
  See \cite{bourgain_nls_radial_gwp, ckstt, ryckman-visan,
    visan_nls_general_d} for the defocusing case and
  \cite{kenig-merle_focusing_nls, kv_focusing_nls} for the focusing
  case.
\end{proof}

As $H = - \Delta + V$ has purely discrete spectrum, global-in-time
spacetime bounds of the
form~\eqref{eqn:nls_wo_potential_spacetime_bound} are not available
even for the linear equation
$i\partial_t u = (-\tfr{1}{2}\Delta + V)u$. Therefore the natural setting
is on a bounded time interval, and we consider

\begin{conjecture}
\label{conj:nls_approx_quadratic_gwp}
When $\mu = 1$, equation \eqref{eqn:nls_approx_quadratic} is
globally wellposed. That is, for each $u_0 \in Q(H)$ there is a
unique global solution $u: \mf{R} \times \mf{R}^d \to \mf{C}$ with
$u(0) = u_0$. This solution obeys the spacetime bound
\begin{equation}
\label{eqn:main_theorem_spacetime_bound}
S_I(u) := \int_I \int_{\mf{R}^d} |u(t,x)|^{\fr{2(d+2)}{d-2}} \, dx \,
dt \le C(|I|, \|u_0\|_{\Sigma} )
\end{equation}
for any compact interval $I \subset \mf{R}$.

If $\mu = -1$, then the same is true provided also that 
\[E(u_0) <
E_{\Delta}(W) \quad \text{and} \quad \| \nabla u_0 \|_{L^2} \le \| \nabla W \|_{L^2}.
\]
\end{conjecture}

The restriction on kinetic energy $\| \nabla u\|_{L^2}$ in the
focusing case is necessary, for as with the harmonic oscillator, we
have:
\begin{thm}
\label{thm:focusing_blowup}
If $\mu = -1, \ E(u_0) < E_{\Delta}(W)$, and $\| \nabla u_0 \|_{L^2}
> \| \nabla W\|_{L^2}$, then the solution to
\eqref{eqn:nls_approx_quadratic} blows up in finite time.
\end{thm}
To prove this one need only make notational changes to
the discussion in \cite[Section~7]{me_quadratic_potential}, and we refer
the reader to there for details.

We state our main result in a conditional fashion to emphasize the
pivotal role of the exactly scale-invariant
problem; by
Theorem~\ref{thm:nls_wo_potential_gwp}, however, the result is
unconditionally valid except in the focusing case for nonradial data in
dimensions $d = 3$ and $4$.

\begin{thm}
\label{thm:nls_approx_quadratic_gwp}
Assume Conjecture~\ref{conjecture:nls_wo_potential_gwp}. Then
Conjecture~\ref{conj:nls_approx_quadratic_gwp} holds.
\end{thm}

NLS with external potentials have both significant physical relevance
(see for example~\cite{zhang_bec}) and mathematical interest as a
dispersive model with broken symmetries. Besides earlier work on the
energy-critical harmonic
oscillator~\cite{me_quadratic_potential,kvz_quadratic_potentials}, we
also mention the papers of
Carles~\cite{carles_time-dependent_potential}, who considered a large
class of subquadratic potentials for the energy-subcritical problem
\[
i\partial_t u = (-\tfr{1}{2}\Delta + V)u + \mu |u|^p u, \ p < \tfr{4}{d-2}.
\]
Taking initial data in $\Sigma$, he established global wellposedness
in the defocusing case when $4/d \le p < 4/(d-2)$ and in the focusing
case when $0 < p < 4/d$. Carles did not require that $V$ be bounded
from below, and also allowed $V = V(t,x)$ to depend on
time. Oh~\cite{oh} had previously proved large data global existence in
the focusing case when $p < 4/d$ and the potential is time-independent
and subquadratic.

We consider a more restricted class of potentials but focus on the
subtleties connected to the energy-critical exponent $p =
4/(d-2)$. When $V = 0$, perturbative arguments and conservation laws
only yield local-in-time solutions whose lifespan depend on the shape
of the initial data, not just on the energy. Thus unlike when
$p < 4/(d-2)$, conservation of energy alone is not sufficient to
preclude finite time blowup.

Although our equation does not actually have scaling symmetry, it
nonetheless contains the same essential difficulties as the
scale-invariant problem. For if we consider initial data of the form
$u_0^{\lambda} = \lambda^{-(d-2)/2} \phi(\lambda^{-1} \cdot)$ for a
fixed Schwartz function $\phi$, and take $\lambda \to 0$, the
energy $E(u_0^{\lambda})$ barely depends on $\lambda$. In
Section~\ref{section:concentrated_initial_data}, we shall see that if
$u^{\lambda}$ is the solution to \eqref{eqn:nls_approx_quadratic} with
$u^\lambda(0) = u_0^\lambda$ and we restrict to a time window $|t|
\lesssim \lambda^2$, then $u^{\lambda}$ can be approximated in
critical spacetime norms by $v^{\lambda}$, where $v^{\lambda}(t, x) =
\lambda^{-(d-2)/2} v(\lambda^{-2}t, \lambda^{-1}x)$ solves the
scale-invariant equation~\eqref{eqn:nls_wo_potential} with $v(0) =
\phi$. Therefore, just as in the scale-invariant case, solutions to
\eqref{eqn:nls_approx_quadratic} with bounded energies can accumulate
nontrivial spacetime norm in arbitrarily short timeframes. 


To prove Theorem~\eqref{eqn:nls_approx_quadratic} we apply the
concentration compactness and rigidity method, which had been adapted
previously to different critical equations
\cite{kenig-merle_focusing_nls,kksv_gkdv,2d_klein_gordon,
  hyperbolic_energy-crit, MR2925134, ionescu,
  kvz_exterior_convex_obstacle}. The reader should also consult the
references following Theorem~\ref{thm:nls_wo_potential_gwp} for the
pioneering instances of this method in scale-invariant
problems. We recall its main ingredients:
\begin{itemize}
\item \textbf{Stability theory}. If $\tilde{u}$ approximately solves
  equation~\eqref{eqn:nls_approx_quadratic} with error sufficiently
  small in Strichartz norms, then there is an \emph{exact} solution
  $u$ to \eqref{eqn:nls_approx_quadratic} with the same initial data
  as $\tilde{u}$, and which is close to $\tilde{u}$ in critical
  spacetime norms.
\item \textbf{Linear and nonlinear profile decompositions}.
  Given a bounded sequence $\{f_n\}_n \subset Q(H)$, there is a
  decomposition $f_n = \sum_j \phi_n^j$, and corresponding
  decompositions of the linear and nonlinear solutions, where the
  profiles are asymptotically pairwise independent and reflect the
  ``symmetries'' of the problem.

\item \textbf{Analysis of scaling limits}. A typical profile in the
  profile decomposition looks schematically like
  $\phi_n = N_n^{(d-2)/2} \phi(N_n \cdot)$ where either $N_n \equiv 1$
  or $\lim_n N_n = \infty$. We will show that in the latter case, for $n$
  large enough the solution $u_n$ to~\eqref{eqn:nls_approx_quadratic}
  with $u_n(0) = \phi_n$ behaves so similarly to a solution to the
  globally wellposed equation~\eqref{eqn:nls_wo_potential} that, by
  stability theory, $u_n$ itself must have finite spacetime norm on a
  length-1 time interval. This essentially rules out blowup for
  equation~\eqref{eqn:nls_approx_quadratic} when the initial data is
  highly concentrated at a point.

\item \textbf{Induction on energy}. Introduced originally by
  Bourgain~\cite{bourgain_nls_radial_gwp} and subsequently refined
  substantially~\cite{ckstt,keraani_nls_blowup,kenig-merle_focusing_nls},
  the idea is
  to assume that global wellposedness
  of~\eqref{eqn:nls_approx_quadratic} fails for some initial data, and
  consider the smallest energy $E_c$ such that solutions $u$ with
  $E(u) \ge E_c$ fail to exist globally. This energy threshold is
  positive by the small data theory. Using the profile decomposition,
  the induction hypothesis that solutions with energy smaller than
  $E_c$ do exist globally, and the scaling limit analysis, one proves
  the existence of a blowup solution $u_c$ with energy $E(u_c) = E_c$,
  and which must simultaneously obey an impossibly strong compactness
  property.

\end{itemize}

In view of the broken translation and scaling symmetry, constructing the
required the profile decompositions is rather involved and
constituted a major component of our previous work on the harmonic
oscillator. We concentrate in this note on the additional ingredients
needed in the present, more general context.
When $H = -\tfr{1}{2}\Delta + \tfr{1}{2}|x|^2$, we exploited at
several junctures the classical Mehler formula for the linear
fundamental solution (see for example~\cite{folland}):
\begin{equation}
\label{eqn:mehler} e^{-itH} (x, y) = \tfr{1}{(2\pi i \sin t)^{d/2}}
e^{\fr{i}{\sin t}(\fr{x^2+y^2}{2} \cos t - xy)}.
\end{equation} No such explicit formula is available for the general
potentials considered in this paper. Instead, we appeal to more robust
microlocal techniques, in particular the
oscillatory integral parametrices 
of
Fujiwara~\cite{fujiwara_fundamental_solution,fujiwara_path_integrals}.

For the more standard arguments, we provide the main steps and
refer the reader to~\cite{me_quadratic_potential} for detailed proofs.

\subsection*{Outline of paper}

In Section~\ref{section:preliminaries} we set our notation and collect
some basic estimates regarding
equation~\eqref{eqn:nls_approx_quadratic}, including Fujiwara's
Fourier integral parametrix. Section~\ref{section:local_theory} states
some standard (but vital) local theory.  The core of this note,
section~\ref{section:conc_compactness}, discusses the profile
decomposition mentioned above. 
The scaling limit analysis of
Section~\ref{section:concentrated_initial_data} and the compactness
arguments of Section~\ref{section:palais-smale} parallel the
ones given in~\cite{me_quadratic_potential}. As will be the case throughout the
paper, we describe mainly the required adjustments and refer to
\cite{me_quadratic_potential} for a comprehensive prsentation.

\subsection*{Acknowledgements}
The author is indebted to his advisors Rowan Killip and Monica Visan
for their helpful discussions as well as their feedback on the
paper.
This work was supported in part by NSF grants DMS-0838680 (RTG),
DMS-1265868 (PI R.~Killip), DMS-0901166, and DMS-1161396 (both PI
M.~Visan).

\section{Preliminaries}
\label{section:preliminaries}

\subsection{Notation and basic estimates}
\label{subsection:notation}

We write $X \lesssim Y$ to mean $X \le C Y$ for some constant
$C$. Similarly $X \sim Y$ means $X \lesssim Y$ and $Y \lesssim
X$. Denote by $L^p(\mf{R}^d)$ the usual Lebesgue spaces, whose norm we
sometimes denote using the compact notation $\|f\|_p$.  If
$I \subset \mf{R}^d$ is an interval, the mixed Lebesgue norms on
$I \times \mf{R}^d$ are defined by
\[
\| f\|_{L^q_t L^r_x (I \times \mf{R}^d)} = \left( \int_I \left( \int_{\mf{R}^d} 
|f(t, x)|^r \, dx \right)^{\fr{q}{r}} dt 
\right)^{\fr{1}{q}}.
\]

We use the following function space notation due to Schwartz
\[
\begin{split}
\mcal{B}_k(\mf{R}^d) &= \{ f \in C^\infty(\mf{R}^d) : D^\ell f \in
L^\infty \ \text{for all} \ \ell \ge k \},\\
\mcal{B}(\mf{R}^d) &= \mcal{B}_0(\mf{R}^d).
\end{split}
\]


\let\oldtheenumi\theenumi
\renewcommand{\theenumi}{(V\arabic{enumi})}

We recall Fujiwara's construction of the fundamental solution
for $H$. Recall that the symbol $H(\xi, x) = \tfr{1}{2}|\xi|^2 +
V(x)$ defines the Hamiltonian flow
\begin{equation}
\label{eqn:hamilton_flow}
\left\{\begin{array}{ll} \dot{x} = \partial_{\xi} H, & x(0) = y\\
         \dot{\xi} =
         -\partial_{x} H, & \xi(0) = \eta.
         \end{array}\right.
\end{equation}
Suppose that $V$ is subquadratic in the sense that
\begin{equation}
  \label{eqn:subquadratic}|V(0)| + |\nabla^k V(x)| \le C_k \text{ for
    all } k
  \ge 2.
\end{equation}
Then the vector field
$(-\partial_x H, \partial_{\xi} H)$ is globally Lipschitz, and we may
regard $x$ and $\xi$ as functions of $(t, y, \eta) \in \mf{R} \times
\mf{R}^d \times \mf{R}^d$.

\begin{prop}[{\cite[Proposition 1.7]{fujiwara_fundamental_solution}}]
\label{prop:classical_trajectories}
  Suppose $V$ satisfies~\eqref{eqn:subquadratic} and put $H(\xi, x) = \tfr{1}{2} |\xi|^2
  + V(x)$. Then the map $(y, \eta) \mapsto (x,
  y)$ obeys the derivative estimates
\[
\tfr{\partial x}{\partial y} = I + t^2 a(t,y,\eta), \quad \tfr{\partial
  x}{\partial \eta} = t(I + t^2 b(t,y,\eta))
\]
for some matrix-valued $a, b \in \mcal{B}(\mf{R}^d_y \times \mf{R}^d_{\eta})$.

Further, there exists $\delta_0$ such that whenever $0 \ne |t| \le
\delta_0$, for pairs $x, y \in \mf{R}^d$ there is a unique trajectory
$(x(\tau), \xi(\tau))$ such that $x(0) = y$ and $x(t) = x$.
\end{prop}
\begin{rmk}
To get the second statement from the first, one invokes the Hadamard
global inverse function theorem to see that $(y, \eta) \mapsto (x, y)$
is a diffeomorphism for $0 \ne t$ sufficiently small.
\end{rmk}

Consequently, when $0 < |t| \le \delta_0$ we can define the \emph{action}
\begin{equation}
\label{eqn:action_def}
S(t, x, y) = \int_0^t \fr{1}{2}|\xi(\tau)|^2 - V(x(\tau)) \, d\tau,
\end{equation}
where $(x(\tau), \xi(\tau))$ is the unique trajectory with $x(0) = y$
and $x(t) = x$.

\begin{thm}[Unitary propagator {\cite{fujiwara_fundamental_solution,fujiwara_path_integrals}}]
\label{thm:fujiwara_propagator}
Let $V$ be subquadratic as in the previous proposition. Then there
exists $\delta_0 > 0$ such that:
\begin{itemize}
\item The action $S(t, x, y)$ is well-defined by~\eqref{eqn:action_def} for all
  $0 < |t| < \delta_0$ and
  satisfies
\begin{equation}
\nonumber
S(t, x, y) = \fr{1}{2t}|x-y|^2 + t\omega(t, x, y),
\end{equation}
where the term $\omega(t, \cdot, \cdot)$ belongs to $\mcal{B}_2$ uniformly for $|t| \le
\delta_0$. That is, there exist constants $C_k$ such that 
\begin{equation}
\nonumber
|\nabla^k_{x,y} \omega(t, x, y)| \le C_k (1 + |x| + |y|)^{ \max(2-k, 0)}
\end{equation}
for all $k$.

\item For all $0 < |t| < \delta_0$ and all $f \in C^\infty_c(\mf{R}^d)$ we
have
\[
e^{-itH} f(x) = \tfr{1}{(2\pi i t)^{d/2}}\int_{\mf{R}^d} e^{iS(t, x, y)}
a(t, x, y) f(y) \, dy,
\]
where
\[
\| \nabla^k_{x, y} [ a(t, \cdot, \cdot) - 1] \|_{L^\infty(\mf{R}^d_x \times
  \mf{R}^d_y)} = O_k(t^2) \quad \text{for all} \quad k \ge 0.
\]

\end{itemize}

\end{thm}
\let\theenumi\oldtheenumi

The above integral representation immediately yields a dispersive estimate:
\begin{cor}[Dispersive estimate]
\label{cor:dispersive_estimate}
For $|t| \le \delta_0$, we have
\[
\| e^{-itH} f\|_{\infty} \lesssim |t|^{-\fr{d}{2}} \| f\|_1.
\]
\end{cor}

A pair of exponents is $(q, r)$ \emph{admissible} if
$2 \le q, r \le \infty$ and $\tfr{2}{q} + \tfr{d}{r} =
\tfr{d}{2}$. For an interval $I$, define the Strichartz spaces
\begin{align*}
  S(I) = L^\infty_t L^2_x  \cap L^2_t L^{\fr{2d}{d-2}}_x (I \times
  \mf{R}^d), \quad N(I) = L^1_t L^2_x + L^2_t L^{\fr{2d}{d+2}}_x (I
  \times \mf{R}^d).
\end{align*}
By interpolation, the $S$ norm controls $\|u\|_{L^q_t L^r_x}$ for all
admissible pairs $(q, r)$, while the $N$ norm is controlled by the
dual $(q', r')$ of any admissible exponents.

\begin{lma}[Strichartz {\cite{keel-tao}}]
\label{lma:strichartz}
Let $I$ be a compact time interval containing $t_0$, and let $u: I
\times \mf{R}^d \to \mf{C}$ be a solution to the inhomogeneous
Schr\"{o}dinger equation 
\[
(i\partial_t - H) u = F.
\]
Then there is a constant $C$, depending only on the length of the
interval $I$, such that 
\[
\| u\|_{S(I)} \le C( \| u_0\|_{L^2} + \| F\|_{N(I)} ).
\]
\end{lma}
\begin{proof}
  This follows from the abstract Keel-Tao theorem \cite{keel-tao} as a
  consequence of the dispersive estimate of the previous corollary,
  and the unitarity of $e^{-itH}$ on $L^2(\mf{R}^d)$.
\end{proof}

As $V$ is nonnegative, we have access to the following spectral
multipler theorem of Hebisch~\cite{hebisch}:
\begin{thm}
\label{thm:hebisch}
If $F: (0, \infty) \to \mf{C}$ is a bounded function which obeys the
derivative estimates
\[
|\partial^k F(\lambda) | \lesssim_k |\lambda|^{-k} \quad \text{for all}
\quad 0 \le k \le \tfr{d}{2} + 1,
\]
then the operator $F(H)$, defined initially on $L^2$ by the Borel
functional calculus, is bounded from $L^p$ to $L^p$ for all $1 < p <
\infty$. 
\end{thm}

The following norm equivalence was first proven for the quadratic
potential by
Killip-Visan-Zhang~\cite[Lemma~2.7]{kvz_quadratic_potentials}. Using
the coercivity hypothesis \ref{eqn:potential_hyp_1}, we adapt their
result to the potentials considered here.

\begin{prop}[Equivalence of norms]
\label{prop:equivalence_of_norms}
For any $1 < p < \infty$ and $s \in [0, 1]$, we have 
\[
\| H^{s} f\|_p \sim_{p, s} \| (-\Delta)^{s} f \|_p + \|
V^s f\|_p
\]
for all Schwartz functions $f$. 
\end{prop}

To prove this we shall need the following fact, which is classical
when $V$ is exactly quadratic; we give a proof for the sake of completeness.
\begin{lma}
\label{lma:smooth_vectors}
Let $H = -\tfr{1}{2}\Delta + V$ where $V \ge 0$ is smooth and satisfies the hypotheses
\ref{eqn:potential_hyp_0}, \ref{eqn:potential_hyp_1}. Then
the space of smooth vectors for $H$ is precisely Schwartz class:
\[
D(H^\infty) := \bigcap_{n \ge 0} D(H^n) = \mcal{S}(\mf{R}^d).
\]
\end{lma}

\begin{proof}[Proof of Proposition~\ref{prop:equivalence_of_norms}]
We show first that 
\begin{equation}
\label{eqn:norm_equivalence_eqn0}
\| (-\Delta)^s f \|_p + \| V^s f\|_p \lesssim_p \| H^s f\|_p \quad
\text{for all} \quad f \in \mcal{S}(\mf{R}^d).
\end{equation}
As $f = H^{-s} H^s f$ and $H^s f \in \mcal{S}(\mf{R}^d)$ by Lemma
\ref{lma:smooth_vectors}, it suffices to prove
\begin{equation}
\label{eqn:norm_equivalence_eqn1}
\| (-\Delta)^s H^{-s} f \|_p  + \| V^s H^{-s} f\|_p \lesssim_p \|f\|_p
\quad \text{for all} \quad f \in \mcal{S}(\mf{R}^d).
\end{equation}
By hypothesis, there is some $\delta > 0$ such that $V(x) \ge \delta
|x|^2$. Killip-Visan-Zhang
\cite{kvz_quadratic_potentials} proved that
\[
\| V^s H^{-s}_{\delta}f \|_p \lesssim_p \| f\|_p,
\] 
where $H_{\delta} = -\tfr{1}{2}\Delta + \delta |x|^2$. On the other
hand, the parabolic maximium principle implies
\[
0 \le e^{-tH} (x, y) \le e^{-tH_{\delta} } (x, y)
\]
Combining this
with the identity
\[
H^{-s}(x, y) = \tfr{1}{\Gamma(s)} \int_0^\infty e^{-tH} (x, y) t^{s-1} dt,
\]
we obtain the kernel inequality
\[
0 \le H^{-s}(x, y) \le H_{\delta}^{-s} (x, y)
\]
In particular, $V^s H^{-s}$ and $V^s H_\delta^{-s}$ have nonnegative
integral kernels. We may therefore bound
\[
\| V^s H^{-s} f\|_{p} \le \| V^s H^{-s} |f| \|_p \le \| V^s
H_{\delta}^{-s} |f| \|_p \lesssim_p \|f \|_p.
\]
This yields half of (\ref{eqn:norm_equivalence_eqn1}). Specializing to
the case $s = 1$ and writing $-\Delta = 2(H - V)$, we
obtain
\begin{equation}
\label{eqn:norm_equivalence_eqn2}
\| (-\Delta) H^{-1} f\|_p \lesssim_p \| f\|_p.
\end{equation}
The rest of the argument is imported directly
from~\cite{kvz_quadratic_potentials}, and is included to make the
discussion self-contained. To show that $\| (-\Delta)^s H^{-s} f\|_p
\lesssim_p \| f\|_p$ for all $s \in [0, 1]$, we use analytic
interpolation.  It suffices to verify that
\[
| \langle (-\Delta)^z H^{-z} f, g \rangle_{L^2} | \le C \|f\|_p \|g\|_{p'}
\]
for all Schwartz $f$ and $g$. By homogeneity, we may assume $\|f\|_p =
\|g\|_{p'} = 1$. Put
\[
F(z) = \langle (-\Delta)^z H^{-z} f, g \rangle_{L^2}.
\]
By the spectral theorm, $F(z)$ is bounded and continuous on the closed
strip $\{ 0 \le \opn{Re}(z) \le 1 \}$ and analytic on its interior. By
the special case (\ref{eqn:norm_equivalence_eqn2}) and 
Theorem \ref{thm:hebisch}, 
\[
\begin{split}
|F(it)| &\le \| (-\Delta)^{it} H^{-it} f\|_p \le C_0\\
|F(1+it)| &\le \| (-\Delta)^{it} (-\Delta) H^{-1} H^{-it} f\|_p \le C_1.
\end{split}
\]
Hadamard's three-lines lemma implies that $|F(z)| \le C$ on the whole strip.
Thus (\ref{eqn:norm_equivalence_eqn0}) and
(\ref{eqn:norm_equivalence_eqn1}) hold for all $p \in (1, \infty)$ and
$s \in [0, 1]$.

Dualizing those estimates yields
\[
\| H^{-s} (-\Delta)^s f\|_p + \| H^{-s} V^s f\|_p \lesssim_{p,s} \|f\|_p
\quad \text{for all} \quad p \in (1, \infty), \ s \in [0, 1]
\]
Writing $H^sf = H^{s-1} Hf = \tfr{1}{2} H^{s-1} (-\Delta)^{1-s}
(-\Delta)^s f + H^{s-1} V^{1-s} V^s f$, we have
\[
\|H^s f\|_p \lesssim_{p,s} \| (-\Delta)^s f\|_p + \| V^s f\|_p \quad
\text{for all} \quad p \in (1, \infty), \ s \in [0, 1].
\]
This completes the proof of the proposition modulo the lemma.
\end{proof}

\begin{proof}[Proof of Lemma~\ref{lma:smooth_vectors}]
  The inclusion $\mcal{S}(\mf{R}^d) \subset D(H^\infty)$ is clear. To prove
  the opposite inclusion, we show by an induction argument the
  equivalent assertion that
\begin{equation}
\label{eqn:smooth_vectors_equiv_inclusion}
D(H^\infty)
  \subset \bigcap_{k \ge 0} \{ u : x^{\alpha} \partial^\beta u \in L^2 \ \text{for all} \ |\alpha|
+ |\beta| \le k\}.
  \end{equation}
We have the following identities:
\begin{equation}
\label{eqn:smooth_vectors_commutators}
\begin{split}
H \partial_j u &= \partial_j H u - (\partial_j V) u\\
H mu &= m H u - \tfr{1}{2}(\Delta m) u - \nabla m \cdot \nabla u
\end{split}
\end{equation}
Define for each $n \ge 1$ the following statements:
\[
\begin{split}
P_1(n) &= ``m: D(H^{n-1}) \to D(H^{n-1}) \ \text{for all} \ m
\in \mcal{B}"\\
P_2(n) &= ``\partial_j: D(H^n) \to D(H^{n-1})"\\
P_3(n) &= ``\partial_j V : D(H^n) \to D(H^{n-1})".
\end{split}
\]
As $D(H) \subset D(H^{1/2}) = \{ u : \|\nabla u\|_{L^2} + \|
xu\|_{L^2} < \infty\}$, these hold for $n = 1$.

Assume that they hold for some $n$. For $u \in D(H^n)$ and $m \in
\mcal{B}$, use \eqref{eqn:smooth_vectors_commutators} and the
statements $P_1(n), \ P_2(n)$ to see that $H (mu) \in D(H^{n-1})$, so
$mu \in D(H^{n})$ and $P_1(n+1)$ holds since $m$ was chosen arbitrarily
in $\mcal{B}$. Similar reasoning shows that $P_2(n)$ and $P_3(n)$
imply $P_2(n+1)$, and that $P_1(n), \ P_2(n), \ P_3(n)$ yield
$P_3(n+1)$. Hence, by induction these statements hold for all $n\ge 1$.

Next, apply~\eqref{eqn:norm_equivalence_eqn1} in the special case
$s = 1, \ p = 2$ to see that 
\[
V : D(H) \to D(H^0) = L^2.
\]
Suppose $u \in D(H^n)$ and $n \ge 2$. We have
\[
H (Vu) = V Hu - \tfr{1}{2}(\Delta V) u - \nabla V \cdot \nabla u.
\]
By induction, $V Hu \in D(H^{n-2})$, while $P_1(n), \ P_2(n)$, and
$P_3(n-1)$ imply that the second and third terms also belong to
$D(H^{n-2})$. Thus $Vu \in D(H^{n-1})$

Summing up, we find that
\[
V : D(H^n) \to D(H^{n-1}) \ \text{for all} \ n \ge 1.
\]
These mapping properties, together with the coercivity hypothesis
\ref{eqn:potential_hyp_1}, immediately yield the
claim~\eqref{eqn:smooth_vectors_equiv_inclusion}.
\end{proof}

Thanks to this norm equivalence, $H^\gamma$ inherits many
properties of the fractional derivative $(-\Delta)^{\gamma}$, including Sobolev embedding:
\begin{lma}[{\cite[Lemma 2.8]{kvz_quadratic_potentials}}]
\label{lma:sobolev_embedding}
Suppose $\gamma \in [0, 1]$ and $1 < p < \fr{d}{2\gamma}$, and define $p^*$
by $\fr{1}{p^*} = \fr{1}{p} - \fr{2\gamma}{d}$. Then
\[
\| f\|_{L^{p^*}(\mf{R}^d)} \lesssim \| H^{\gamma} f\|_{L^p(\mf{R}^d)}.
\]
\end{lma}

Similarly, the fractional chain and product rules carry over to the 
current setting: 

\begin{cor}[{\cite[Proposition~2.10]{kvz_quadratic_potentials}}]
\label{cor:fractional_chain_rule}
Let $F(z) = |z|^{\fr{4}{d-2}} z$. For any $0 \le \gamma \le \fr{1}{2}$ and 
$1 < p < \infty$,
\[
\| H^{\gamma} F(u)\|_{L^p(\mf{R}^d)} \lesssim \| F'(u)\|_{L^{p_0}(\mf{R}^d)} \| 
H^{\gamma}f \|
_{L^{p_1}(\mf{R}^d)}
\]
for all $p_0, p_1 \in (1, \infty)$ with $p^{-1} = p_0^{-1} +
p_1^{-1}$. 
\end{cor}

Using Proposition~\ref{prop:equivalence_of_norms} and the Christ-Weinstein
fractional product rule for $(-\Delta)^\gamma$
(e.g. \cite{taylor_tools}), we obtain
\begin{cor}
\label{cor:fractional_product_rule}
For $\gamma \in (0, 1], \  r, p_i, q_i \in (1, \infty)$ with $r^{-1} = p_i^{-1} +
q_{i}^{-1}, \ i = 1, 2$, we have 
\[
\| H^{\gamma} (fg)\|_r \lesssim \| H^{\gamma} f\|_{p_1} \|g\|_{q_1} +
\|f\|_{p_2} \|H^{\gamma} g\|_{q_2}.
\]
\end{cor}

\subsection{FIO technology}
\label{subsection:fio}

We review some properties of Fourier integral operators tailored to the
Schr\"{o}dinger equation, which were developed 
by Fujiwara \cite{fujiwara_oscillatory_integrals} and Asada-Fujiwara
\cite{asada-fujiwara}. 

\begin{define}
\label{def:fio_tech_phase}
A \emph{phase function} is a smooth $\phi(x, y) \in
\mcal{B}_2(\mf{R}^d_x \times \mf{R}^d_y)$ which satisfies the
nondegeneracy condition
\begin{equation}
\label{eqn:fio_tech_nondegeneracy}
\inf_{x, y} | \det \nabla^2_{xy} \phi (x, y)| > 0.
\end{equation}

\end{define}
Given a phase $\phi(x, y)$ and an amplitude
$a(x, y) \in \mcal{B}( \mf{R}^d_x\times \mf{R}^d_y)$, define for each
$\lambda \ne 0$ the integral operator
\begin{equation}
\label{eqn:fio_tech_op_defn}
A(\lambda) f(x) = \left (\fr{\lambda}{2\pi i} \right )^{\fr{d}{2}} \int_{\mf{R}^d}
e^{i\lambda \phi(x, y)} a(x, y) f(y)dy.
\end{equation}
Note that in this notation, which we have carried over
from~\cite{asada-fujiwara}, $\lambda$ plays the role of frequency or
equivalently the inverse of the semiclassical parameter.
\begin{rmk}
Asada and Fujiwara studied more general
oscillatory integral operators of the form
\[
f \mapsto  \left(\fr{\lambda}{2\pi i}\right)^{\fr{m+n}{2}} \int_{\mf{R}^m}
\int_{\mf{R}^n} e^{i\phi(x, \theta, y)} a(x, \theta, y) f(y) \, dy d\theta
\]
where $a(x, \theta, y) \in \mcal{B}(\mf{R}^n_x \times
\mf{R}^m_{\theta} \times \mf{R}^n_y)$ and the phase $\phi$ satisfies
a nondegeneracy condition
\[
\left |\det \left(\begin{array}{cc} \nabla^2_{xy} \phi & \nabla^2_{x\theta}
      \phi\\ 
\nabla^2_{\theta y} \phi & \nabla^2_{\theta \theta} \phi \end{array}\right)
\right | \ge \delta 
\]
\end{rmk}

\begin{thm}[Fujiwara \cite{fujiwara_oscillatory_integrals}]
\label{thm:fujiwara_bddness}
$\| A(\lambda)\|_{L^2 \to L^2} \le C \| a\|_{C^{2d+1}}$
\end{thm}

Let $\phi$ be a phase function. By the global inverse
function theorem, the maps 
\[\chi_1 (x, y) = (y, -\partial_y \phi) \quad \text{and} \quad \chi_2(x, y) =
(x, \partial_x \phi )
\]
are diffeomorphisms of
$\mf{R}^d \times \mf{R}^d$. It follows that the relation
\[
(y, -\partial_y \phi ) \mapsto (x, y) \mapsto (x, \partial_x \phi)
\]
defines a diffeomorphism 
\[
\chi = \chi_2 \circ \chi_1^{-1}: \mf{R}^d_y \times \mf{R}^d_\eta \to
\mf{R}^d_x \times \mf{R}^d_\xi,
\]
which preserves the standard symplectic form $d\xi \wedge dx$.  The
map $\chi(y, \eta) = (x(y, \eta), \xi(y,\eta))$ is the canonical
transformation generated by the phase function $\phi(x,y)$.

For a smooth symbol $p \in \mcal{B}_k(\mf{R}^d_x \times
\mf{R}^d_{\theta} \times \mf{R}^d_{y})$ and $\lambda \ne 0$, let
$\opn{Op}(p, \lambda)$ denote the (semiclassical) pseudodifferential operator 
\[
\opn{Op}(p, \lambda) f(x) = \left(\fr{\lambda}{2\pi}\right)^d \iint
e^{i\lambda (x-y) \theta} p(x,
\theta, y) f(y) \, dy d\theta.
\]
One has an Egorov theorem:
\begin{thm}[{\cite[Theorem 6.1]{asada-fujiwara}}]
\label{thm:egorov}
Let
$\chi : \mf{R}^{2d} \to \mf{R}^{2d}$ be the canonical transformation
generated by a phase function $\phi(x,y)$, and consider a Fourier
integral operator $A(\lambda)$ with phase $\phi$ and amplitude $a$. Let
$p(x, \theta, y), \ q(x, \theta, y) \in \mcal{B}_1(\mf{R}^d \times
\mf{R}^d \times \mf{R}^d)$ be such that 
\[
q(y, \eta, y) = p(x, \xi, x) | _{(x, \xi) = \chi(y, \eta)}
\]
Then 
\[
\opn{Op}(\lambda p, \lambda) A(\lambda) - A(\lambda) \opn{Op}(\lambda q,\lambda) =  R(\lambda),
\]
for some Fourier integral operator  $R(\lambda)$ with phase function
$\phi$. The operator norm of $R(\lambda)$ satisfies
\[
\|R(\lambda)\|_{L^2 \to L^2} \lesssim \lambda^{-1} \| a\|_{0, M} (
\|p\|_{1, M} + \|q\|_{1, M})
\]
for some positive integer $M$, where
$
\|f\|_{r, s} = \sup_{r \le k \le s} \| \nabla^k f\|_{L^\infty}.
$
\end{thm}

\section{Local Theory}
\label{section:local_theory}

We record some standard local-wellposedness results for
\eqref{eqn:nls_approx_quadratic}. These are immediate analogues of the
local theory for the scale-invariant equation
\eqref{eqn:nls_wo_potential}, as detailed in the
lecture notes~\cite{claynotes}. Essentially the same proofs work here with the help of~\ref{lma:sobolev_embedding} and
Corollaries \ref{cor:fractional_chain_rule},
\ref{cor:fractional_product_rule}.

\begin{prop}[Local wellposedness]
\label{prop:lwp}
Let $u_0 \in \Sigma(\mf{R}^d)$ and fix a compact time interval $0 \in I \subset 
\mf{R}$. Then there 
exists a constant $\eta_0 = \eta_0(d, |I|)$ such that whenever $\eta < \eta_0$ 
and 
\[
\| H^{\fr{1}{2}} e^{-itH} u_0 \|_{L^{\fr{2(d+2)}{d-2}}_t 
L^{\fr{2d(d+2)}{d^2+4}}_x (I \times \mf{R}^d)} 
\le \eta,
\]
there exists a unique solution $u: I \times \mf{R}^d \to \mf{C}$ to 
\eqref{eqn:nls_approx_quadratic} which satisfies the bounds
\begin{align*}
\| H^{\fr{1}{2}}  u \|_{L^{\fr{2(d+2)}{d-2}}_t L^{\fr{2d(d+2)}{d^2+4}}_x (I 
\times \mf{R}^d)} \le 2\eta \quad \text{and} \quad
\| H^{\fr{1}{2}} u \|_{S(I)} \lesssim \| u_0\|_{\Sigma} + 
\eta^{\fr{d+2}{d-2}}.
\end{align*}
\end{prop}

\begin{cor}[Blowup criterion]
\label{cor:blowup_criterion}
Suppose $u : (T_{\text{min}}, T_{\text{max}}) \times \mf{R}^d \to \mf{C}$ is a 
maximal lifespan 
solution to \eqref{eqn:nls_approx_quadratic}, and fix $T_{\text{min}} < t_0 < T_{\text{max}}$. 
If $T_{\text{max}} < 
\infty$, then 
\[
\| u\|_{L^{\fr{2(d+2)}{d-2}}_{t, x}( [t_0, T_{\text{max}} ))} = \infty.
\]
If $T_{\text{min}} > -\infty$, then
\[
\| u \|_{L^{\fr{2(d+2)}{d-2}}_{t, x} ( (T_{\text{min}}, t_0])} = \infty.
\]
\end{cor}

\begin{prop}[Stability]
\label{prop:stability}
Fix $t_0 \in I \subset \mf{R}$ an interval of unit length and let $\tilde{u}: 
I \times \mf{R}^d \to 
\mf{C}$ be 
an approximate solution to \eqref{eqn:nls_approx_quadratic} in the 
sense that 
\[
i\partial_t \tilde{u} = H u \pm |\tilde{u}|^{\fr{4}{d-2}} \tilde{u} + e
\]
for some function $e$. Assume that 
\begin{equation}
\label{eqn:stability_prop_hyp_1} 
\| \tilde{u}\|_{L^{\fr{2(d+2)}{d-2}}_{t,x}} \le L,
\quad \| H^{\fr{1}{2}}u \|_{L^\infty_t L^2_x} \le E,
\end{equation}
and that for some $0 < \varepsilon < \varepsilon_0(E, L)$ one has
\begin{equation}
\label{eqn:stability_prop_hyp_2} 
\| H^{1/2}(\tilde{u}(t_0) - u_0) \|_{L^2} + \| H^{\fr{1}{2}} e \|_{N(I)} \le \varepsilon,
\end{equation}
Then there exists a unique solution $u : I \times \mf{R}^d 
\to \mf{C}$ to \eqref{eqn:nls_approx_quadratic} with $u(t_0) = u_0$ 
and which further satisfies 
the estimates
\begin{equation}
\| \tilde{u} -  u \|_{L^{\fr{2(d+2)}{d-2}}_{t,x} } +
\| H^{\fr{1}{2}} (\tilde{u} - u) \|_{S(I)} \lesssim C(E, L) 
\varepsilon^c
\end{equation}
where $0 < c = c(d) < 1$ and $C(E, L)$ is a function which is nondecreasing
in each variable.
\end{prop}

\section{Concentration compactness}
\label{section:conc_compactness}




Let $0 \in I$ be a compact interval so that $|I| \le \delta_0$, where
$\delta_0$ is the constant in Theorem~\ref{thm:fujiwara_propagator}.
As a basic building block in our analysis, we need suitable profile
decompositions for the linear and nonlinear equations. The discussion
here focuses on the linear case which already contains most of the
subtleties. In view of the perturbative theory in
Section~\ref{section:local_theory}, we seek to characterize
initial data with nontrivial linear evolutions, i.e. which come close
to saturating the the Strichartz inequality
\[
\| e^{-itH} f\|_{L^{\fr{2(d+2)}{d-2}}_{t,x} (I \times \mf{R}^d)}
\lesssim \| H^{1/2} f\|_{L^2}.
\]

A substantial part of our previous work on the harmonic oscillator was
devoted to constructing profile decompositions for
$H = -\tfr{1}{2}\Delta + \tfr{1}{2}|x|^2$. We closely follow that
exposition but highlight a key technical difference in the
present setting. As
alluded to in the introduction, we must compare the linear evolutions
of a highly concentrated initial state under the propagators
$e^{-itH}$ and $e^{\fr{it\Delta}{2}}$ with and without a potential,
respectively (see Proposition~\ref{prop:strong_convergence}
below). For the harmonic oscillator we relied on the Mehler formula to
write
\[
e^{-itH} = m_t(x) e^{i\fr{i\sin(t)\Delta}{2}} m_t(x)
\]
where $m_t(x) = \exp (i(\tfr{\cos t - 1}{2\sin t}) x^2)$, which
clearly manifests the relation between the two propagators. Here, we
shall instead appeal to the general parametrix in
Theorem~\eqref{thm:fujiwara_propagator} and apply the estimates from
Section~\ref{subsection:fio}.



\begin{define}
\label{def:frame}
A \emph{frame} is a sequence $(t_n, x_n, N_n) \in I \times \mf{R}^d
\times 2^{\mf{N}}$ conforming to one of the following scenarios:
\begin{enumerate}
\item \label{enum:frame_1} $N_n \equiv 1, \ t_n \equiv 0$, and $x_n \equiv
  0$.
\item \label{enum:frame_2} $N_n \to \infty$ and $ N_n^{-1}
  V(x_n)^{1/2} \to r_\infty \in [0, \infty)$.
\end{enumerate}
\end{define}

\begin{rmk}
The quantity $N_n^{-1} V(x_n)^{1/2}$ is the analog of the ratio $N_n^{-1}|x_n|$
that was considered in~\cite{me_quadratic_potential}.
\end{rmk}

These parameters will specify the temporal center, spatial center, and
(inverse) length scale of a function. The hypothesis that $V$ grows
essentially quadratically ensures that $|x_n| \lesssim N_n$, which
reflects the fact that we only consider functions obeying some uniform
bound in $Q(H)$, and such functions cannot be centered arbitrarily far
from the origin. We need to augment the frame $\{(t_n, x_n, N_n)\}$
with an auxiliary parameter $N_n'$, which corresponds to a sequence of
spatial cutoffs adapted to the frame.
\begin{define}
\label{def:aug_frame}
An \emph{augmented frame} is a sequence $(t_n, x_n, N_n, N_n') \in I \times \mf{R}^d
\times 2^{\mf{N}} \times \mf{R}$ belonging to one of the following types:
\begin{enumerate}
\item \label{enum:aug_frame_1} $N_n \equiv 1, \ t_n \equiv 0, \ x_n \equiv
  0, \ N_n' \equiv 1$.
\item \label{enum:aug_frame_2} $N_n \to \infty, \ N_n^{-1}
  V(x_n)^{1/2} \to r_\infty \in [0, \infty)$, and either
\begin{enumerate}
\item \label{enum:aug_frame_2_1} $N_n' \equiv 1$ if  $r_\infty > 0$, or
\item \label{enum:aug_frame_2_2} 
$N_n^{1/2} \le N_n' \le N_n, \ N_n^{-1} V(x_n)^{1/2} (\tfr{N_n}{N_n'})
\to 0$, and $\tfr{N_n}{N_n'} \to \infty$ if $r_\infty = 0$.
\end{enumerate}
\end{enumerate}
\end{define}

Given an augmented frame $(t_n, x_n, N_n, N_n')$, we define scaling and
translation operators on functions of space and of spacetime by
\begin{equation}
\label{eqn:inv_strichartz_frameops_1}
\begin{split}
(G_n \phi)(x) &= N_n^{\fr{d-2}{2}} \phi(N_n(x-x_n))\\
 (\tilde{G}_n f)(t, x) &= N_n^{\fr{d-2}{2}} f(N_n^2(t - t_n), N_n(x-x_n)).
\end{split}
\end{equation}
We also define spatial cutoff operators $S_n$ by
\begin{equation}
\label{eqn:inv_strichartz_frameops_2}
S_n \phi = \left\{ \begin{array}{cc} \phi, & \text{for frames of
      type} \ \ref{enum:aug_frame_1} \quad (\text{i.e.} \ N_n \equiv 1),\\
\chi(\tfr{N_n}{N_n'} \cdot) \phi, &  \text{for
  frames of type} \ \ref{enum:aug_frame_2} \quad (\text{i.e.} \ N_n \to \infty),
\end{array}\right.
\end{equation}
where $\chi$ is a smooth compactly supported function equal to $1$ on
the ball $\{|x| \le 1\}$. The following mapping properties of these
operators are elementary:
\begin{equation}
\label{eqn:inv_strichartz_frameops_properties}
\begin{split}
\lim_{n \to \infty} S_n = I \ \text{strongly in} \ \dot{H}^1 \text{and in} \ Q(H),\\
\limsup_{n \to \infty} \| G_n\|_{Q(H) \to Q(H)} < \infty.
\end{split}
\end{equation}

The next technical lemma is the counterpart of
\cite[Lemma~\ref{quadratic-lma:approximation_lemma}]{me_quadratic_potential}
and is proved in the same manner (in particular we use the equivalence of norms
furnished by Proposition~\ref{prop:equivalence_of_norms}).
\begin{lma}[Approximation]
\label{lma:approximation_lemma}
Let $(q, r)$ be an admissible pair of exponents with $2 \le r <
d$, and let $\mcal{F} = \{(t_n, x_n, N_n, N_n')\}$ be an augmented
frame of type 2.
\begin{enumerate}
\item Suppose $\mcal{F}$ is of type \ref{enum:aug_frame_2_1} in Definition
  \ref{def:aug_frame}. Then for $\{f_n\} \subseteq L^q_t
  H^{1, r}_x(\mf{R} \times \mf{R}^d)$, we have
\[
\limsup_{n} \| H^{1/2} \tilde{G}_n S_n f_n \|_{L^q_t L^r_x} \lesssim
\limsup_{n} \| f_n \|_{L^q_t H^{1,r}_x}. 
\]
\item Suppose $\mcal{F}$ is of type \ref{enum:aug_frame_2_2} and $f_n \in L^q_t
  \dot{H}^{1,r}_x (\mf{R} \times \mf{R}^d)$. Then
\[
\limsup_{n} \| H^{1/2} \tilde{G}_n S_n f_n \|_{L^q_t L^r_x} \lesssim 
\limsup_{n} \| f_n \|_{L^q_t \dot{H}^{1,r}_x}.
\]
\end{enumerate}
Here $H^{1, r}(\mf{R}^d)$ and $\dot{H}^{1, r}(\mf{R}^d)$ denote the
inhomogeneous and homogeneous $L^r$ Sobolev spaces, respectively,
equipped with the norms 
\[
\| f\|_{H^{1, r}} = \| \langle \nabla \rangle\|_{L^r(\mf{R}^d)}, \quad
\|f\|_{\dot{H}^{1, r}} = \| |\nabla| f\|_{L^r(\mf{R}^d)}.
\]
\end{lma}

We come to the main results of this section.
\begin{prop}[Inverse Strichartz]
\label{prop:inv_strichartz}
Let $I$ be a compact interval containing $0$ of length at most $\delta_0$, and
suppose $f_n$ is a sequence of functions in $Q(H)$ satisfying
\[
0 < \varepsilon \le \| e^{-itH} f_n\|_{L^{\fr{2(d+2)}{d-2}}_{t,x}(I
  \times \mf{R}^d)}
\lesssim \|H^{1/2}f_n\|_{L^2} \le A < \infty.
\]
Then, after passing to a subsequence, there exists an augmented frame 
\[
\mcal{F} =
\{ (t_n, x_n, N_n, N_n') \}
\]
and a sequence of functions $\phi_n\in
Q(H)$ such that one of the following holds:
\begin{enumerate}
\item \label{enum:inv_strichartz_case_1} $\mcal{F}$ is of type
  \ref{enum:aug_frame_1} (i.e. $N_n \equiv 1$) and $\phi_n = \phi$ where $\phi \in Q(H)$
  is a weak limit of $f_n$ in $Q(H)$.
\item \label{enum:inv_strichartz_case_2} $\mcal{F}$ is of type
  \ref{enum:aug_frame_2}, either $t_n \equiv 0$ or $N_n^2 t_n \to \pm
  \infty$, and $\phi_n = e^{it_n H} G_n S_n \phi$
  where 
$\phi \in \dot{H}^1(\mf{R}^d)$ is a weak limit of $G_n^{-1} e^{-it_n H} f_n$ in
$\dot{H}^1$. Moreover, if $\mcal{F}$ is of type
\ref{enum:aug_frame_2_1}, then $\phi$ also belongs to $L^2(\mf{R}^d)$.
\end{enumerate}
The functions $\phi_n$ have the following properties:
\begin{gather}
 \liminf_{n} \| H^{1/2} \phi_n \|_{L^2} \gtrsim A \left( \tfr{\varepsilon}{A} 
\right)^{\fr{d(d+2)}{8}} \label{eqn:inv_strichartz_nontriviality_in_Sigma}\\
 \lim_{n \to \infty} \|f_n\|_{L^{\fr{2d}{d-2}}}^{\fr{2d}{d-2}} - \|f_n - 
\phi_n\|_{L^{\fr{2d}{d-2}}}^{\fr{2d}{d-2}} - \| \phi_n 
\|_{L^{\fr{2d}{d-2}}}^{\fr{2d}{d-2}}  = 0. \label{eqn:inv_strichartz_potential_energy_decoupling}\\
 \lim_{n \to \infty}  \|H^{1/2}f_n\|_{L^2}^2 - \|  H^{1/2}(f_n - \phi_n)\|_{L^2}^2 - 
\| H^{1/2}\phi_n \|_{L^2}^2  = 0 \label{eqn:inv_strichartz_decoupling_in_Sigma}
\end{gather}

\end{prop}

\begin{proof}
We recall that the proof of the analogous result
in~\cite[Section~\ref{quadratic-subsection:inv_strichartz}] {me_quadratic_potential} used the following ingredients:
\begin{itemize}
\item Littlewood-Paley theory adapted to the operator
  $H = - \tfr{1}{2}\Delta + \tfr{1}{2} |x|^2$, which depended on a
  spectral multiplier theorem (Theorem~\ref{thm:hebisch}).
  
\item A refined Strichartz inequality, proved using the Littlewood-Paley theory.
\item Convergence properties of \emph{equivalent} and
  \emph{orthogonal} frames, in particular, the comparison of the
  linear flows generated by the Hamiltonians for the free particle and
  the harmonic oscillator, when acting on concentrated initial
  data. It was here that we invoked the Mehler
  formula~\eqref{eqn:mehler}.

\end{itemize}

Once suitable analogues for these components are obtained, the rest of
the proof carries over without difficulty, and we refer the reader
to~\cite{me_quadratic_potential} for the details. Adapting the first
two to our situation requires little more than replacing all instances
of $\tfr{1}{2}|x|^2$ in the proofs with $V$. The third requires
elaboration, however, and will be the subject of the next section.
\end{proof}

\begin{prop}[Linear profile decomposition]
\label{prop:lpd}
Let $0 \in I$ be an interval with $|I| \le \delta_0$, and let $f_n$ be
a bounded sequence in $Q(H)$. After passing to a subsequence, there
exists $J^* \in \{0, 1, \dots\} \cup \{\infty\}$ such that for each
finite $1 \le j \le J^*$, there exist an augmented frame $\mcal{F}^j =
\{(t_n^j, x_n^j, N_n^j, (N_n^j)')\}$ and a function $\phi^j$ with the
following properties.
\begin{itemize}
\item Either $t_n^j \equiv 0$ or $(N_n^j)^2 (t_n^j) \to \pm \infty$ as
  $n \to \infty$.
\item $\phi^j$ belongs to $Q(H), \ H^1$, or $\dot{H}^1$ depending on
  whether $\mcal{F}^j$ is of type \ref{enum:aug_frame_1},
  \ref{enum:aug_frame_2_1}, or \ref{enum:aug_frame_2_2},
  respectively. 
\end{itemize}
For each finite $J \le J^*$, we have a decomposition 
\begin{equation}
\label{eqn:lpd_decomposition}
f_n = \sum_{j=1}^J e^{it_n^j H} G_n^j S_n^j \phi^j + r_n^J,
\end{equation}
where $G_n^j, \ S_n^j$ are the $\dot{H}^1$-isometry and spatial cutoff
operators associated to $\mcal{F}^j$. Writing $\phi^j_n$ for $e^{it_n^j H}
G_n^j S_n^j \phi^j$, this decomposition has the following properties:
\begin{gather}
\label{eqn:lpd_bubble_maximality}
(G_n^J)^{-1} e^{-it_n^J H} r_n^J \overset{\dot{H}^1}{\rightharpoonup}
0 \quad \text{for all} \ J \le J^*,\\
\label{eqn:lpd_sigma_decoupling}
\sup_{J} \lim_{n \to \infty} \Bigl | \| H^{1/2} f_n \|_{L^2}^2 - \sum_{j=1}^J \|H^{1/2}
   \phi^j_n\|_{L^2}^2 - \| H^{1/2} r_n^J\|_{L^2}^2 \Bigr | = 0,\\
\label{eqn:lpd_potential_energy_decoupling}
\sup_{J} \lim_{n \to \infty} \Bigl | \|f_n\|_{L^{\fr{2d}{d-2}}_x}^{\fr{2d}{d-2}}
  - \sum_{j=1}^J \| \phi^j_n\|_{L^{\fr{2d}{d-2}}_x}^{\fr{2d}{d-2}} -
  \| r_n^J \|_{L^{\fr{2d}{d-2}}_x}^{\fr{2d}{d-2}} \Bigr | = 0.
\end{gather}
Whenever $j \ne k$, the frames $\{ (t_n^j, x_n^j, N_n^j)\}$ and $\{ (t_n^k, x_n^k,
N_n^k)\}$ are orthogonal:
\begin{equation}
\label{eqn:lpd_orthogonality_of_frames}
\lim_{n \to \infty} \tfr{N_n^j}{N_n^k} + \tfr{N_n^k}{N_n^j} + N_n^j N_n^k |t_n^j
- t_n^k| + \sqrt{N_n^j N_n^k} |x_n^j - x_n^k| = \infty.
\end{equation}
Finally, we have
\begin{equation}
\label{eqn:lpd_vanishing_of_remainder}
\lim_{J \to J^*} \limsup_{n \to \infty} \| e^{-it_n H} r_n^J
\|_{L^{\fr{2(d+2)}{d-2}}_{t,x}} = 0,
\end{equation}
\end{prop}

\begin{proof}
  The argument is similar to the one for as
  in~\cite[Proposition~\ref{quadratic-prop:lpd}]{me_quadratic_potential}. One
  inductively applies inverse Strichartz to extract the frames
  $\mcal{F}^j$ and profiles $\phi^j$. To prove the decoupling
  assertion~\eqref{eqn:lpd_orthogonality_of_frames}, one uses the
  convergence lemmas discussed in the next section, which completely
  parallel the ones used in~\cite{me_quadratic_potential}.
\end{proof}

\begin{rmk}
When $H = -\tfr{1}{2}\Delta + \tfr{1}{2}|x|^2$, one can
improve~\eqref{eqn:lpd_sigma_decoupling} to
\[
\begin{split}
&\Bigl | \|\nabla f_n \|_{L^2}^2 - \sum_{j=1}^J \| \nabla \phi_n^j \|_{L^2}^2 -
\| \nabla r_n^J\|_{L^2}^2 \Bigr | \to 0\\
 &\Bigl | \|x f_n \|_{L^2}^2 - \sum_{j=1}^J \| x \phi_n^j \|_{L^2}^2 -
\| x r_n^J\|_{L^2}^2 \Bigr | \to 0.
\end{split}
\]
The proof of this relies on the fact that $e^{-itH}$ conjugates $\nabla$
and $x$ according to Heisenberg's equations. We used this stronger
assertion to treat the focusing equation in an earlier draft
of~\cite{me_quadratic_potential}, but were unable to extend it to the
present setting.  It would be interesting to know whether one has
separate decoupling of kinetic and potential energies for more general
potentials.
\end{rmk}

\subsection{Convergence of linear propagators}
\label{subsection:convergence}

In this section we prove the key
Proposition~\ref{prop:strong_convergence}, which compares the linear
propagators $e^{\fr{it\Delta}{2}} \phi$ and $e^{it (\fr{\Delta}{2} - V)}
\phi$ for $\phi$ highly concentrated.


While the proposition is simply a translation of
\cite[Lemma~\ref{quadratic-lma:strong_convergence}]{me_quadratic_potential},
its proof is more involved and requires a closer study of the
underlying classical dynamics.

\begin{define}
\label{def:equiv_frames}
We say two frames $\mcal{F}^1 = \{ (t_n^1, x_n^1, N_n^1)\}$ and $\mcal{F}^2 =
\{ (t_n^2, x_n^2, N_n^2)\}$ (where the superscripts are indices, not
exponents) are \emph{equivalent} if
\[
\tfr{N_n^1}{N_n^2} \to R_\infty \in (0, \infty), \ N_n^1(x_n^2 - x_n^1) \to
x_\infty \in \mf{R}^d, \  (N_n^1)^2 (t_n^1 - t_n^2) \to
t_\infty \in \mf{R}.
\]
If any of the above statements fail, we say that $\mcal{F}_1$ and $\mcal{F}_2$
are \emph{orthogonal}. Note that replacing the  $N_n^1$ in the second and third 
expressions above by $N_n^2$ yields an equivalent definition of
orthogonality.

Two augmented frames $(t_n, x_n, N_n, N_n')$ and $(\tilde{t}_n,
\tilde{x}_n, \tilde{N}_n, \tilde{N}_n')$ are said to be equivalent if their underlying
frames $(t_n, x_n, N_n)$ and $(\tilde{t}_n, \tilde{x}_n, \tilde{N}_n)$
are equivalent.
\end{define}

\begin{prop}[Strong convergence]
\label{prop:strong_convergence}
Suppose $\mcal{F}^M = (t_n^M, x_n, M_n)$ and $\mcal{F}^N = (t_n^N,
y_n, N_n)$ are equivalent frames. Define
\begin{alignat*}{2}
&R_\infty = \lim_{n \to \infty} \tfr{M_n}{N_n} &\quad &t_\infty = \lim_{n \to
  \infty} M_n^2 (t_n^M - t_n^N), \\
 &x_\infty = \lim_{n \to \infty} \ M_n(y_n-x_n), &\quad &r_\infty = \lim_n M_n^{-1} V(x_n)^{1/2};
\end{alignat*}
(The last limit exists by the definition of a frame.) Let $G_n^M, G_n^N$ be the scaling and translation operators associated
with the frames $\mcal{F}^M$ and $\mcal{F}^N$ respectively. Then the sequence
$(e^{-it_n^N H} G_n^N)^{-1} e^{-it_n^M H} G_n^M$ converges in the
strong operator topology on $B(\Sigma, \Sigma)$ to the
operator $U_\infty$ defined by
\[
U_\infty \phi = e^{-it_\infty (r_{\infty})^2}
R_\infty^{\fr{d-2}{2}} [e^{ \fr{it_\infty \Delta}{2} } \phi]( R_\infty
\cdot + x_\infty).
\]
\end{prop}

\begin{proof}

Write $(e^{-it_n^N H} G_n^N)^{-1} e^{-it_n^M H} G_n^M = (G_n^N)^{-1}
G_n^M (G_n^M) e^{-it_n H} G_n^M$ where $t_n = t_n^M - t_n^M$. As
$(G_n^N)^{-1} G_n^M$ converges strongly to the operator $f \mapsto
R_\infty^{\fr{d-2}{2}} f ( R_\infty \cdot + x_\infty)$, it suffices to
show that 
\begin{equation}
\label{eqn:strong_conv_reduced}
(G_n^M)^{-1} e^{-it_n H} G_n^M \to e^{-it_\infty (r_\infty)^2}
e^{\fr{it_\infty \Delta}{2}}.
\end{equation}

Recall from
Theorem~\ref{thm:fujiwara_propagator} that the phase in the Fourier
integral formula for $e^{-itH}$ is the classical action and has the form
\[
S(t, x, y) = \tfr{|x-y|^2}{2t} + t \omega(t, x, y), \quad \omega(t, \cdot,
\cdot) \in \mcal{B}_2.
\]
We first refine this asymptotic to identify the limit of the sequence
and then establish convergence using the oscillatory integral
theory of Section~\ref{subsection:fio}.

The leading terms of the action are obtained by replacing the
classical trajectories with straight lines
in the integral~\eqref{eqn:action_def}. Proceeding in the spirit of Fujiwara~\cite{fujiwara_fundamental_solution}, we have the following lemma.

\begin{lma}
\label{lma:refined_action}
Let $H(\xi, x) = \tfr{1}{2}|\xi|^2 + V(x)$ with $V$ subquadratic, and
let $S(t, x, y)$ be the action (which is well-defined for all $x$ and
$y$ so long as $|t| \le \delta_0$ where $\delta_0$ is the constant in Theorem~\ref{thm:fujiwara_propagator}). Then
\[
S(t, x, y) = \fr{|x-y|^2}{2t} - \int_0^t V(y+ (\tfr{x-y}{t}) \tau) \,
d\tau + O \Bigl (t^3 (1+|x|^2 + |y|^2) \Bigr ).
\]
\end{lma}

\begin{proof}
  The system \eqref{eqn:hamilton_flow} may be written equivalently as
  \begin{equation}
\label{eqn:integrated_hamilton_flow}
\begin{split}
\xi(t) &= \eta - \int_0^t \partial_x V (x(\theta)) \, d\theta, \\
x(t) &= y + \int_0^t \xi(\tau) d\tau = y + t\eta - \int_0^t
(t-\theta) \partial_x V (x(\theta) ) \, d\theta.
\end{split}
\end{equation}
As $\partial_xV$ grows at most linearly, Gronwall's inequality implies
that for all initial data $y, \eta$ we have
\begin{equation}
\nonumber
|x(t)| \le C (1 + |y| + |t\eta|).
\end{equation}
Fix a time $t > 0$ and positions $x, y \in \mf{R}^d$. By Proposition
\ref{prop:classical_trajectories}, there is a unique initial velocity
$\eta = \eta(t, x, y)$ such that the solution $(x(\tau), \xi(\tau) )$
to \eqref{eqn:integrated_hamilton_flow} satisfies $x(0) = y$ and $x(t) = x$.

Referring to the definition~\eqref{eqn:action_def} of the action, we
estimate the error incurred by replacing the true trajectory by the
straight line path from $y$ to $x$.
Rearranging the above expression for $x(t)$, we have 
\begin{equation}
\label{eqn:refined_action_initial_momentum}
\eta = \fr{x-y}{t} + \fr{1}{t}  \int_0^t (t - \theta) \partial_x
V(x(\theta)) \, d\theta.
\end{equation}
For $\tau$ between $0$ and $t$,
\[
|x(\tau)| \le |y| + \Bigl |\fr{x-y}{t} \tau \Bigr |  + C \int_0^t |t-\theta| (1 +
|x(\theta)|) \, d\theta,
\]
hence $|x(\tau)| \le C(1 + |x| + |y|)$.
The preceding computations reveal that
\[
\begin{split}
\Bigl |x(\tau) - y - \tau (\fr{x-y}{t}) \Bigr | &\le \fr{\tau}{t} \int_0^t |t-\theta|
|\partial_x V (x(\theta) )| \, d\theta + \int_0^\tau |\tau-\theta|
|\partial_x V (x(\theta))| \, d\theta\\
&\le C (\tau t + \tau^2) (1 + |x| + |y|).
\end{split}
\]
By the fundamental theorem of calculus, 
\begin{equation}
\label{eqn:refined_action_potential_error}
\begin{split}
\int_0^t |V(x(\tau)) - V(y + \tau (\tfr{x-y}{t}))| \, d \tau &\le C \int_0^t (
\tau t + \tau^2) (1 + |x| + |y|)^2 d\tau \\
&\le C t^3 (1 + |x| + |y|)^2.
\end{split}
\end{equation}

Next, by combining the first line of
\eqref{eqn:integrated_hamilton_flow} with
\eqref{eqn:refined_action_initial_momentum}, we find that
\[
\xi(\tau) = \fr{x-y}{t}  + \fr{1}{t} \int_0^t (t-\theta) \partial_x V
(x(\theta) ) \, d\theta - \int_0^\tau \partial_x V ( x(\theta) ) \, d\theta.
\]
It is easy to see that second and third terms are bounded by $O(t(1+|x|
+ |y|))$. Therefore,
\[
\begin{split}
\int_0^t \fr{1}{2} |\xi(\tau)|^2 \, d\tau &= \fr{|x-y|^2}{2t} +
\fr{x-y}{t} \int_0^t (t - \theta) \partial_x V (x(\theta)) \, d\theta
\\
&- \fr{x-y}{t} \int_0^t \int_0^\tau \partial_x V (x(\theta)) \, d\theta
d\tau  + O(t^3 (1+|x| + |y|)^2)\\
&= \fr{|x-y|^2}{2t} + O(t^3(1+|x|+|y|)^2).
\end{split}
\]
Combining this with \eqref{eqn:refined_action_potential_error} establishes
the lemma.
\end{proof}

By Theorem~\ref{thm:fujiwara_propagator} and a change of variable,
\begin{equation}
\label{eqn:strong_conv_kernel}
\begin{split}
(G_n^M)^{-1} e^{-it_n H} G_n^M f(x) 
= \left(\fr{\lambda_n}{2\pi i} \right)^{\fr{d}{2}} \int_{\mf{R}^d}
e^{i\lambda_n \phi_n(x, y)} a_n(x, y) f(y) \, dy,
\end{split}
\end{equation}
where
\begin{align*}
\lambda_n &= (M_n^2 t_n)^{-1}\\
a_n(x, y) &=  a(t_n, x_n + M_n^{-1} x, x_n + M_n^{-1} y)\\
\phi_n(x, y) &= \tfr{1}{2} |x-y|^2 + \lambda_n^{-1} t_n \omega(t_n, x_n
+ M_n^{-1} x, x_n + M_n^{-1} y)\\
&= \phi_0(x, y) + \lambda_n^{-1} t_n \omega_n(x, y).
\end{align*}
Theorem~\ref{thm:fujiwara_propagator} and 
Lemma~\ref{lma:refined_action} imply that these quantities
obey the following estimates:
\begin{equation}
\label{eqn:strong_conv_estimates}
\begin{split}
&t_n \omega_n(x, y) = -\int_0^{t_n} V(x_n +
M_n^{-1} y + \tfr{x-y}{M_n t_n}\tau) \, d\tau + O( t_n^3 (|x_n|^2 +
M_n^{-2} |x|^2 + M_n^{-2}|y|^2))\\
&= - t_n V_n(x_n) + O(
M_n^{-2} (1 + |x|^2 + |y|^2) ),\\
&|\nabla^k_{x, y} \omega_n (x, y)| \lesssim \left\{\begin{array}{cc}
M_n^{-1} (1 + |x_n + M_n^{-1} x| + |x_n + M_n^{-1} y|), & k = 1\\
M_n^{-k}, & k \ge 2\end{array}\right.\\
&|\nabla^m_{x, y} [a_n(x,y)-1] | \lesssim_k M_n^{-2-k} \quad \text{for all} \quad k
\ge 0.\\
\end{split}
\end{equation}

We need the following adaptation of \cite[Proposition~4.15]{fujiwara_fundamental_solution}.
\begin{lma}
\label{lma:strong_conv_unif_bddness}
The operators $(G_n^M)^{-1} e^{-it_n H} G_n^M$ are uniformly bounded on $\Sigma$.
\end{lma}

\begin{proof}
  Let $\chi_n: (y, -\partial_y \phi_n ) \mapsto (x, \partial_x \phi_n)$
be the canonical transformation
generated by the phase function $\phi_n$. In terms of the variables
$(y, \eta)$, we have
\[
\begin{split}
\chi_n(y, \eta) &= (y + \eta, \eta) + \lambda_n^{-1} t_n ( \partial_y
\omega_n, \partial_x \omega_n + \partial_y \omega_n)(x(t_n, y, \eta), y)\\
&= (y + \eta, \eta) + (r_{1, n}(y, \eta), r_{2, n}(y, \eta)).
\end{split}
\]

First we show that 
\begin{equation}
\label{eqn:strong_conv_eqn1}
\| \nabla (G_n^M)^{-1} e^{-it_n H} G_n^M f\|_{L^2} \lesssim \| f\|_{\Sigma}.
\end{equation}
Put $p(x, \theta, y) = \theta$ and $q_n(x, \theta, y) = \theta +
r_{2, n}(y, \theta)$. By construction,
\[
p(x, \xi, x)|_{(x, \xi) = \chi_n(y, \eta)} = q_n(y, \eta, y).
\]
By the representation \eqref{eqn:strong_conv_kernel} and Theorem~\ref{thm:egorov}, 
\begin{equation}
\label{eqn:strong_conv_commutator}
\begin{split}
D (G_n^M)^{-1} e^{-it_n H} G_n^M &=  \opn{Op}(\lambda_n p,  \lambda_n)
(G_n^M)^{-1} e^{-it_n H} G_n^M \\
&= (G_n^M)^{-1} e^{-it_n
  H} G_n^M \opn{Op}(\lambda_n q_n, \lambda_n) + R_n(\lambda_n),
\end{split}
\end{equation}
where we write $D = \tfr{1}{i} \nabla$. In light of the estimates \eqref{eqn:strong_conv_estimates} and
Theorem~\ref{thm:fujiwara_bddness}, it suffices to obtain a uniform bound
\begin{equation}
\nonumber
\|\opn{Op}(\lambda_n q_n, \lambda_n)\|_{ \Sigma \to L^2} \lesssim 1
\end{equation}
By definition
\[
\opn{Op}(\lambda_n q_n, \lambda_n)f (x) = Df + \opn{Op}(\lambda_n r_{2, n}, \lambda_n).
\]
Using \eqref{eqn:strong_conv_estimates} and Proposition
\ref{prop:classical_trajectories}, we see that
\[
\begin{split}
&\lambda_n r_{2, y}(y, \eta) = t_n(\partial_x\omega_n + \partial_y
\omega_n )(t_n, x(t_n, 0, 0), 0)\\
&+ t_n y \int_0^1 (\partial^2_{xy} \omega_n) (t_n, x(t_n, sy, s\eta), sy)
\tfr{\partial x}{\partial y} + (\partial^2_{y} \omega_n) (t_n, x(t_n, sy, s\eta), sy) \, ds\\
&+ t_n \eta \int_0^1 (\partial^2_{xy} \omega_n) (t_n, x(t_n, sy, s\eta), sy)
\tfr{\partial x}{\partial \eta} \, ds\\
&= c_n  + y r_{2,n}^1(y, \eta) + \eta r_{2,n}^2(y, \eta),
\end{split}
\]
where $|c_n| \lesssim M_n^{-2}$ and $\| D^k r_{2,n}^1 \|_{L^\infty}
\lesssim M_n^{-4}, \ \|
D^k r_{2,n}^2\|_{L^\infty} \lesssim M_n^{-6}$ for all $k$. Thus
\[
\begin{split}
&\opn{Op}(\lambda_n r_{2,n}, \lambda_n) = c_n I + \opn{Op}( y
r_{2,n}^1(y, \eta), \lambda_n ) + \opn{Op}(\eta r_{2,n}^2 (y, \eta), \lambda_n )\\
&= c_n I + \opn{Op}(r_{2,n}^1, \lambda_n) X +
\opn{Op}(\lambda_n^{-1} r_{2,n}^2,
\lambda_n) D + \opn{Op}(\lambda_n^{-1} (D_y r_{2,n}^{2}), \lambda_n).
\end{split}
\]
The Calder\'{o}n-Vaillancourt theorem now
implies 
\[
\| \opn{Op}(\lambda_n r_{2,n}, \lambda_n) f\|_{L^2} \lesssim
 M_n^{-2} \|f\|_{L^2} + M_n^{-4}\|xf\|_{L^2} + M_n^{-6} \|D f\|_{L^2} \lesssim M_n^{-2}\|f\|_{\Sigma}.
\]
Altogether we obtain \eqref{eqn:strong_conv_eqn1}. 

By setting $p(x, \theta, y) = x, \ q(x, \theta, y) = y + \theta
+ r_{1, n}(y, \eta)$ and making a similar analysis as above, we obtain
\[
\| x(G_n^M)^{-1} e^{-it_n H} G_n^M f\|_{L^2} \lesssim \| f\|_{\Sigma}.
\]
This concludes the proof of the lemma.
\end{proof}

We now verify the limit~\eqref{eqn:strong_conv_reduced}. As $e^{\fr{iM_n^2 t_n \Delta}{2}} \to
e^{\fr{it_\infty \Delta}{2}}$ strongly, it suffices to show that
\[
(G_n^M)^{-1} e^{-it_n H} G_n^M f  - e^{-i t_\infty (r_\infty)^2}
  e^{\fr{iM_n^2 t_n\Delta}{2}} f
\]
converges to $0$ for all $f \in \Sigma$.
By
Lemma \ref{lma:strong_conv_unif_bddness} we may assume $f
\in C^\infty_c$. 
The above difference may be written as
\[
\begin{split}
  \left(\tfr{\lambda_n}{2\pi
      i}\right)^{\fr{d}{2}} \int e^{i\lambda_n
    \phi_n} [a_n - 1] f(y) \, dy &+ \left( \tfr{\lambda_n}{2\pi i} \right)^{\fr{d}{2}} \int [
  e^{i\lambda_n \phi_n} -  e^{-it_\infty (r_\infty)^2}e^{ i \lambda_n \phi_0}] f(y) \, dy\\
    &= A_n f + B_n f.
\end{split}
\]
Using Theorem \ref{thm:fujiwara_bddness} and the estimates \eqref{eqn:strong_conv_estimates}, one argues as in the
proof of Lemma \ref{lma:strong_conv_unif_bddness} to see that
$\|A_nf\|_{\Sigma} \lesssim M_n^{-2}\|f\|_{\Sigma}$. 

It remains to bound $B_nf$. By hypothesis $f$ is supported in some
ball $B(0, R)$, and the estimates~\eqref{eqn:strong_conv_estimates}
show that the integral kernel of $B_n$ converges to $0$ in $C^\infty_{loc}$. It follows that $|xB_n f|$ and $|\nabla B_n
f|$ converge to $0$ locally uniformly. On the other hand, 
integration by parts reveals that for all $n$ sufficiently
large,
\[
|xB_nf| + |\nabla B_nf | \lesssim_N |x|^{-N}
\]
for any $N>0$ and for all $|x| \ge 4R$. Hence $\|B_nf\|_{\Sigma} \to
0$ by dominated convergence. This completes the proof of the
proposition.
\end{proof}

In the remainder of this section we collect other lemmata
regarding equivalent and orthogonal frames. They can be proved in much
the same manner as their counterparts in
\cite[Section~\ref{quadratic-subsection:convergence}]{me_quadratic_potential}.

\begin{cor}
\label{cor:strong_conv_cor}
Let $\{ (t_n^M, x_n, M_n, M_n')\}$ and $\{ t_n^N, y_n, N_n, N_n')\}$
be equivalent augmented frames. Let $S_n^M, \ S_n^N$ be the associated
spatial cutoff operators. Then

\begin{equation}
\label{eqn:strong_conv_cor_eqn1}
\lim_{n \to \infty} \| e^{-it_n^M H} G_n^M S_n^M \phi - e^{-it_n^N H}
G_n^N S_n^N U_\infty \phi
\|_{\Sigma} = 0
\end{equation}
and
\begin{equation}
\label{eqn:strong_conv_cor_eqn2}
\lim_{n\to \infty} \| e^{-it_n^M H} G_n^M S_n^M \phi - e^{-it_n^N H} G_n^N U_\infty S_n^N \phi
\|_{\Sigma} = 0
\end{equation}
whenever $\phi \in H^1$ if the frames conform to case
\ref{enum:aug_frame_2_1} and $\phi \in \dot{H}^1$ if they conform to
case \ref{enum:aug_frame_2_2} in Definition \ref{def:aug_frame}. 
\end{cor}

\begin{proof}
Run an approximation argument using
Lemma~\ref{lma:approximation_lemma} in the manner of \cite[Corollary~\ref{quadratic-cor:strong_conv_cor}]{me_quadratic_potential}.
\end{proof}

The following ``approximate adjoint'' identity is the analogue
of~\cite[Lemma~\ref{quadratic-lma:approximate_adjoint}]{me_quadratic_potential}.
\begin{lma}
\label{lma:approximate_adjoint}
Suppose the frames $\{(t_n^M, x_n, M_n)\}$ and $\{ (t_n^N,
y_n, N_n)\}$ are equivalent. Put $t_n = t_n^M - t_n^N$. Then for $f, g
\in \Sigma$ we have
\[
 \langle (G_n^N)^{-1} e^{-it_n H} G_n^M f, g \rangle_{\dot{H}^1} = \langle 
f, (G_n^M)^{-1} e^{it_n H} G_n^N g \rangle_{\dot{H}^1} + R_{n} (f, g),
\]
where $|R_n (f, g)| \le C|t_n| \| G_n^M f\|_{\Sigma} \| G_n^N g\|_{\Sigma}$.  
\end{lma}

\begin{proof}
The proof of Lemma \ref{lma:strong_conv_unif_bddness} yields the
following commutator estimate:
\[
\| [D, e^{-itH}] \|_{\Sigma \to L^2} = O(t).
\]
We have
\[
\begin{split}
&\langle D(G_n^N)^{-1} e^{-it_n H} G_n^M f, D g \rangle_{L^2} = \langle Df, D(G_n^M)^{-1} e^{it_n H} G_n^N g\rangle_{L^2} + R_n(f, g)
\end{split}
\]
where $R_n(f, g) = \langle [D,
e^{-it_n H}] G_n^M f, D G_n^N g \rangle_{L^2} - \langle D G_n^M f, [D,
e^{it_n H} ] G_n^N g \rangle_{L^2}$. The claim then follows from
Cauchy-Schwarz and the above estimate.
\end{proof}

The next lemma is a converse to Proposition~\ref{prop:strong_convergence}.
\begin{lma}[Weak convergence]
\label{lma:weak_convergence}
Assume the frames $\mcal{F}^M = \{ (t_n^M, x_n, M_n)\}$ and $\mcal{F}^N
= \{ (t_n^N, y_n, N_n) \}$ are orthogonal. 
Then, for any $f \in \Sigma$,
\[
(e^{-it_n^N H} G_n^N)^{-1} e^{-it_n^M H} G_n^M f \to 0 \quad \text{weakly in} \quad \dot{H}^1.
\]

\end{lma}

\begin{proof}


  Put $t_n = t_n^M - t_n^N$, and suppose that $|M_n^2 t_n|
  \to \infty$. Then 
  \[
  \| (G_n^N)^{-1} e^{-it_n H} G_n^M
  f\|_{L^{\fr{2d}{d-2}}} \to 0
  \]
  for $f \in C^\infty_c$ by a change of
  variables and the dispersive estimate, thus for general $f \in
  \Sigma$ by a density argument. Therefore $(G_n^N)^{-1} e^{-it_n H}
  G_n^M f$ converges weakly in $\dot{H}^1$ to $0$. We consider next
  the case where $M_n^{2} t_n \to t_\infty \in \mf{R}$. The
  orthogonality of $\mcal{F}^M$ and $\mcal{F}^N$ implies that either
  $N_n^{-1} M_n$ converges to $0$ or $\infty$, or $M_n |x_n-y_n|$
  diverges as $n \to \infty$. In either case, one verifies easily that
  the operators $(G_n^N)^{-1} G_n^M$ converge to zero in the weak
  operator topology on $B(\dot{H}^1, \dot{H}^1)$. Applying
  Proposition~\ref{prop:strong_convergence}, we see that $(G_n^N)^{-1}
  e^{-it_n H} G_n^M f = (G_n^N)^{-1} G_n^M (G_n^M)^{-1} e^{-it_n H}
  G_n^M f$ converges to zero weakly in~$\dot{H}^1$.
\end{proof}

\begin{cor}
\label{cor:weak_conv_cor}
Let $\{(t_n^M, x_n, M_n, M_n')\}$ and $ \{ (t_n^N, y_n, N_n,
N_n')\}$ be augmented frames such that $\{(t_n^M, x_n, M_n)\}$ and
$\{(t_n^N, y_n, N_n)\}$ are orthogonal. Let $G_n^M, \ S_n^M$ and $G_n^N, S_n^N$ be the associated operators.
Then 
\[
(e^{-it_n^N H} G_n^N)^{-1} e^{-it_n^M H} G_n^M S_n^M \phi \rightharpoonup 0 \quad
\text{in} \quad \dot{H}^1
\]
whenever $\phi \in H^1$ if $\mcal{F}^M$ is of type
\ref{enum:aug_frame_2_1} and $\phi \in \dot{H}^1$ if $\mcal{F}^M$ is
of type \ref{enum:aug_frame_2_2}.
\end{cor}

\begin{proof}
If $\phi \in C^\infty_c$, then $S^M_n \phi = \phi$ for all large $n$, and the claim
follows from Lemma~\ref{lma:weak_convergence}. The case of general $\phi$ in
$H^1$ or $\dot{H}^1$ then follows from an approximation argument
similar to the one used in the proof of
Corollary~\ref{cor:strong_conv_cor}. 
\end{proof}

\section{The case of concentrated initial data}
\label{section:concentrated_initial_data}

With the main complications out of the way, we sketch the
rest of the wellposedness argument in the remaining two sections. The
next step is to rule out blowup for
equation~\eqref{eqn:nls_approx_quadratic} when the initial data is
highly concentrated in space.

\begin{prop}
\label{prop:localized_initial_data}
Let $I = [-\delta_0/2, \delta_0/2]$, where  $\delta_0$ is the constant in
Theorem~\ref{thm:fujiwara_propagator}. Assume that
Conjecture~\ref{thm:nls_wo_potential_gwp} holds. Let
  \[
  \mcal{F} = \{ (t_n, x_n, N_n, N_n')\}
  \]
  be an augmented frame with $t_n \in I$ and $N_n \to \infty$, such
  that either $t_n \equiv 0$ or $N_n^2 t_n \to \pm \infty$; that is,
  $\mcal{F}$ is type \ref{enum:aug_frame_2_1} or
  \ref{enum:aug_frame_2_2} in Definition \ref{def:aug_frame}. Let
  $G_n, \tilde{G}_n$, and $S_n$ be the associated operators as defined
  in \eqref{eqn:inv_strichartz_frameops_1} and
  \eqref{eqn:inv_strichartz_frameops_2}. Suppose $\phi$ belongs to
  $H^1$ or $\dot{H}^1$ depending on whether $\mcal{F}$ is type
  \ref{enum:aug_frame_2_1} or \ref{enum:aug_frame_2_2}
  respectively. Then, for $n$ sufficiently large, there is a unique
  solution $u_n : I \times \mf{R}^d \to \mf{C}$ to the defocusing
  equation~\eqref{eqn:nls_approx_quadratic}, $\mu = 1$, with initial
  data
\[
u_n(0)=  e^{it_n H} G_n S_n \phi.
\]
This solution satisfies a spacetime bound
\[
\limsup_{n \to \infty} S_I(u_n) \le C( E(u_n)).
\]
Suppose in addition that $\{ (q_k, r_k) \}$ is any finite collection
of admissible pairs with $2 < r_k < d$. Then for each $\varepsilon >
0$ there exists $\psi^{\varepsilon} \in C^\infty_c(\mf{R} \times
\mf{R}^d)$ such that
\begin{equation}
\label{eqn:approximation_by_smooth_functions}
\limsup_{n \to \infty} \sum_k \| H^{1/2}(u_n - \tilde{G}_n [e^{-itN_n^{-2}V(x_n)}  \psi^{\varepsilon}]) \|_{L^{q_k}_t
  L^{r_k}_x (I \times \mf{R}^d)} < \varepsilon.
\end{equation}

Assuming also that $\| \nabla \phi\|_{L^2} < \| \nabla W\|_{L^2}$ and
$E_{\Delta}(\phi) < E_{\Delta}(W)$, we have the same conclusion as above for the
focusing equation~\eqref{eqn:nls_approx_quadratic}, $\mu =
-1$.
\end{prop}

\begin{proof}[Proof sketch]
  We only give a rough idea as one can proceed just as in
  Proposition~\ref{quadratic-section:localized_initial_data} of
  \cite{me_quadratic_potential} and replace every instance of
  $\tfr{1}{2}|x_n|^2$ wih $V(x_n)$. The idea is to show that for $n$
  large enough, one can fashion a sufficiently accurate approximate
  solution $\tilde{u}_n$ on the interval $I$ in the sense of
  Proposition~\ref{prop:stability}, such that $S_I(\tilde{u}_n)$ are
  bounded. This bound will then be transferred to the exact solution
  $u_n$ by the stability theory.

  While $u_n$ remains highly concentrated (over time scales on the
  order of $N_n^{-2}$), it will be approximated by a modified
  solution to the scale-invariant
  equation~\eqref{eqn:nls_wo_potential} (whose solutions admit global
  spacetime bounds). By the time this
  approximation breaks down, the solution $u_n$ will have dispersed to
  such an extent that the evolution of $u_n$ is essentially
  linear.

  If $t_n \equiv 0$, let $v$ be the global
  solution to \eqref{eqn:nls_wo_potential} furnished by
  Conjecture~\ref{conjecture:nls_wo_potential_gwp} with $v(0) =
  \phi$. If $N_n^2 t_n \to \pm \infty$, let $v$ be the (unique)
  solution to~\eqref{eqn:nls_wo_potential} which scatters in
  $\dot{H}^1$ to $e^{\fr{it\Delta}{2}} \phi$ as $t \to \mp
  \infty$. Note the reversal of signs.

  The approximate solution is defined as follows. Let $\tilde{G}_n$
  and $S_n$ be the operators defined in
  \eqref{eqn:inv_strichartz_frameops_1} and
  \eqref{eqn:inv_strichartz_frameops_2}, and define for each $n$ a
  Littlewood-Paley cutoff
\[
P_{\le
  \tilde{N}_n'} = \varphi(-\Delta / (\tilde{N}_n')^2), \quad \tilde{N}_n' =
(\tfr{N_n}{N_n'})^{\fr{1}{2}},
\]
where $\varphi : \mf{R} \to \mf{R}$
denotes a smooth function equal to $1$ on the ball $B(0, 1)$ and
supported in $B(0, 1.1)$. Fix a large $T > 0$, and define
\begin{equation}
\label{eqn:localized_initial_data_v_n-tilde_defn}
\tilde{v}_n^T(t) = \left\{ \begin{array}{cc}  e^{ -itV(x_n)} \tilde{G}_n[
    S_n P_{\le \tilde{N}_n'} v] (t+t_n)  & |t| \le TN_n^{-2}\\
e^{-i(t-TN_n^{-2}) H} \tilde{v}_n^T(TN_n^{-2}), & TN_n^{-2} \le t \le \delta_0\\
e^{-i(t+TN_n^{-2})H} \tilde{v}_n^T(-TN_n^{-2}), & -\delta_0 \le t \le -TN_n^{-2}
\end{array}\right..
\end{equation}

Inside the ``window of concentration'', $\tilde{v}_n^T$ is
essentially a modulated solution to \eqref{eqn:nls_wo_potential} with cutoffs
applied in both space, to place the solution in $C_t \Sigma_x$, and
frequency, to enable taking an extra derivative in the error analysis
for the stability theory. The time translation by $t_n$ is needed to undo
the time translation built into the operator $\tilde{G}_n$; see
\eqref{eqn:inv_strichartz_frameops_1}. 

Essentially the same computations as
in~\cite{me_quadratic_potential} yield the estimate
\begin{gather*}
\limsup_n \| H^{1/2} \tilde{v}_n^T \|_{L^\infty_t L^2_x([-\delta_0, \delta_0])}
+ \| \tilde{v}_n^T \|_{L^{\fr{2(d+2)}{d-2}}_{t,x}([-\delta_0,
  \delta_0] \times \mf{R}^d)} \lesssim C(\| \phi\|_{\dot{H}^1}),
\end{gather*}
uniformly in $T$; one also sees that
\begin{gather*}
\lim_{T \to \infty} \limsup_{n} \| H^{1/2}[ (i\partial_t -
H)(\tilde{v}_n^T) - F(\tilde{v}_n^T)] \|_{N([-\delta_0, \delta_0])} =
0,
\end{gather*}
where $F(z) = \mu |z|^{\fr{4}{d-2}} z$ is the nonlinearity. 
\[
\lim_{T \to \infty} \limsup_{n} \|H^{1/2}[ \tilde{v}_n^T (-t_n) -
u_n(0)] \|_{L^2_x} = 0.
\]
Thus, for some fixed large $T$ and all large $n$,
$\tilde{u}_n (t, x) := \tilde{v}_n^T(t - t_n, x)$ is an approximate
solution on the time interval $[-\delta_0/2, \delta_0/2]$ in the sense
of Proposition~\ref{prop:stability}. Thus one obtains the first part
of Proposition~\ref{prop:localized_initial_data}. The last claim
regarding approximation by smooth functions is proven by applying
Lemma~\ref{lma:approximation_lemma} to the functions $\tilde{v}_n^T$
in the manner of
\cite[Lemma~\ref{quadratic-lma:approximation_by_smooth_functions}]{me_quadratic_potential}. 
\end{proof}

\section{A compactness property for blowup sequences}
\label{section:palais-smale}

In this section we give a Palais-Smale condition on blowup sequences
of solutions to \eqref{eqn:nls_approx_quadratic}. 
This will quickly lead to the proof of Theorem~\ref{thm:nls_approx_quadratic_gwp}. 

For a maximal solution $u$ to \eqref{eqn:nls_approx_quadratic}, 
define 
\[
S_*(u) = \sup \{ S_I(u) : I \text{ is an open interval with } |I| \le 1\},
\]
where we set $S_I(u) = \infty$ if $u$ is not 
defined on $I$.
Set 
\[
\begin{split}
\Lambda_d(E) &= \sup \{ S_*(u) : u \text{ solves \eqref{eqn:nls_approx_quadratic}, } \mu = +1, \ E(u) = E\}\\
\Lambda_f(E) &= \sup \{ S_*(u) : u \text{ solves
  \eqref{eqn:nls_approx_quadratic}, } \mu = -1, \ E(u) = E, \\
&\| \nabla u(0)\|_{L^2} < 
\| \nabla W \|_{L^2}\}.
\end{split}
\]
Finally, define
\[
\begin{split}
\mcal{E}_d &= \{ E : \Lambda_d(E) < \infty\}, \quad
\mcal{E}_f 
= \{ E : \Lambda_f(E) < \infty\},
\end{split}
\]
By the local theory,
Theorem~\ref{thm:nls_approx_quadratic_gwp} is equivalent to the
assertions
\[
\mcal{E}_d = [0, \infty), \quad \mcal{E}_f = [0, E_{\Delta}(W) ).
\]

Suppose Theorem~\ref{thm:nls_approx_quadratic_gwp} failed. By the small data theory, 
$\mcal{E}_d, \ \mcal{E}_f$ are nonempty and open, and the failure of 
Theorem~\ref{thm:nls_approx_quadratic_gwp} implies the existence of a critical 
energy $E_c > 0$, with $E_c < E_{\Delta}(W)$ in the focusing case such 
that 
$\Lambda_d(E), \ \Lambda_f (E) = \infty$ for $E > E_c$ and
$\Lambda_d(E), \ \Lambda_f(E) < \infty$ for all $E < E_c$. 
We have the following compactness property.


\begin{prop}[Palais-Smale]
\label{prop:palais-smale}
Assume Conjecture~\ref{thm:nls_wo_potential_gwp} holds. Suppose that $u_n :
(t_n - \delta_0, t_n + \delta_0 ) \times \mf{R}^d \to \mf{C}$ is a
sequence of solutions with
\[
\lim_{n \to \infty} E(u_n) =E_c, \ \lim_{n \to \infty} S_{(t_n-\delta_0, t_n]} (u_n) = 
\lim_{n \to \infty} S_{ [t_n, t_n+\delta_0)} (u_n) = \infty,
\]
where $\delta_0$ is the constant in
Theorem~\ref{thm:fujiwara_propagator}. In the focusing case, assume
also that $E_c < E_{\Delta}(W)$ and $\| \nabla u_n (t_n)\|_{L^2} < \|
\nabla W\|_{L^2}$. Then there exists a subsequence such that
$u_n(t_n)$ converges in $ Q(H)$.
\end{prop}

\begin{proof}
We refer to the presentation following Proposition~\ref{quadratic-prop:palais-smale}
in~\cite{me_quadratic_potential}. The
proof uses a local smoothing estimate for the propagator $e^{-itH}$,
which can be obtained via a multiplier argument just as in
Corollary~\ref{quadratic-cor:local_smoothing_cor}
of~\cite{me_quadratic_potential}. In the focusing case, one also uses
energy trapping arguments (see
Section~\ref{quadratic-section:focusing_blowup} of
\cite{me_quadratic_potential}) to see that the hypotheses are in
fact equivalent to $\| H^{1/2} u_n(t_n) \|_{L^2} < \| \nabla W\|_{L^2}$.
\end{proof}

\begin{proof}[Proof of Theorem~\ref{thm:nls_approx_quadratic_gwp}]
  Suppose the theorem failed, and let $E_c$ be as above. Then, after
  applying suitable time translations, there is
  a sequence of solutions $u_n$ with $ E(u_n) \to E_c$ and $
  S_{(-\delta_0 / 4, \delta_0/4)}(u_n) \to \infty.  $ Choose $t_n$
  such that $S_{(-\delta_0 / 4, t_n)} (u_n) = \tfr{1}{2} S_{(-\delta_0
    / 4, \delta_0 / 4)} (u_n)$. By
  Proposition~\ref{prop:palais-smale}, after passing to a subsequence
  we have $\| u(t_n) - \phi\|_{\Sigma} \to 0$ for some $\phi \in
  \Sigma$. Then $E(\phi) = \lim_{n} E(u_n(t_n)) = E_c$.

  Let $v: (-T_{min}, T_{max}) \to \mf{C}$ be the maximum-lifespan solution
  to~\eqref{eqn:nls_approx_quadratic} with $v(0) = \phi$. By comparing
  $v(t,x)$ with the solutions $u_n(t + t_n, x)$ and applying
  Proposition~\ref{prop:stability}, we see that $S_{(0, \delta_0 /
    2)}(v) = S_{(-\delta_0/2, 0)}(v) = \infty$. Thus $-\delta_0/2 \le
  -T_{min} < T_{max} \le \delta_0/2$. But the orbit $\{v(t)\}_{t \in
    (-T_{min}, T_{max})}$ is a precompact subset of $\Sigma$, by
  Proposition~\ref{prop:palais-smale}, so there is some sequence of
  times $t_n$ increasing to $T_{max}$ such that $v(t_n)$ converges in
  $\Sigma$ to some $\psi$. By considering a local solution with
  initial data $\psi$ and invoking stability theory, we see that $v$
  can actually be extended to some larger interval $(-T_{min}, T_{max}
  + \eta)$, in contradiction to the maximality of $v$.
\end{proof}

\bibliographystyle{abbrv}

\bibliography{bibliography}

\end{document}